\documentclass[12pt]{amsart}
\usepackage[utf8]{inputenc} 
\usepackage[active]{srcltx}
\usepackage{a4wide}
\usepackage{amsthm,amsfonts,amsmath,mathrsfs,amssymb}
\usepackage{dsfont}
\usepackage{mathtools}
\usepackage[T1]{fontenc}
\usepackage[utf8]{inputenc}
\usepackage{enumerate}
\usepackage{comment}
\usepackage[left=4cm,top=4cm,right=4cm,bottom=4cm]{geometry}
\numberwithin{equation}{section}

\usepackage{geometry}
\usepackage{graphicx}
\usepackage{amssymb}
\usepackage{amsmath}
\usepackage{amsthm}
\usepackage{listings}
\usepackage[colorlinks=true,urlcolor=blue,citecolor=blue,linkcolor=blue]{hyperref}
\usepackage{esint}
\usepackage{xcolor}

\usepackage{etex}
\usepackage[all]{xy}
\usepackage{pgf,tikz}
\usetikzlibrary{arrows}
\usepackage{pgfplots,pgfcalendar}
\usetikzlibrary{patterns}
\usepackage{hyperref}
\def\a{\alpha}       \def\b{\beta}

\def\D{{\mathbb D}}  \def\T{{\mathbb T}}
\def\C{{\mathbb C}}  \def\N{{\mathbb N}}

\def\R{\mathbb R}

\def\({\left(}       \def\){\right)}

\def\Re{{\sf Re}\,}
\def\Im{{\sf Im}\,}

\newtheorem{theorem}{Theorem}[section]
\newtheorem{lemma}[theorem]{Lemma}
\newtheorem{proposition}[theorem]{Proposition}
\newtheorem{corollary}[theorem]{Corollary}

\theoremstyle{definition}
\newtheorem{definition}[theorem]{Definition}
\newtheorem{example}[theorem]{Example}

\theoremstyle{remark}
\newtheorem{remark}[theorem]{Remark}

\numberwithin{equation}{section}

\DeclareMathOperator*{\esssup}{ess\,sup}
\begin{document}
\title[Integration operators]{Integration operators in average radial integrability spaces of analytic functions}
\author[T. Aguilar-Hern\'andez ]{Tanaus\'u Aguilar-Hern\'andez}
\address{Departamento de Matem\'atica Aplicada II and IMUS, Escuela T\'ecnica Superior de Ingenier\'ia, Universidad de Sevilla,
	Camino de los Descubrimientos, s/n 41092, Sevilla, Spain}
\email{taguilar@us.es}

\author[M.D. Contreras]{Manuel D. Contreras}
\address{Departamento de Matem\'atica Aplicada II and IMUS, Escuela T\'ecnica Superior de Ingenier\'ia, Universidad de Sevilla,
	Camino de los Descubrimientos, s/n 41092, Sevilla, Spain}
\email{contreras@us.es}

\author[L. Rodr\'iguez-Piazza]{Luis Rodr\'iguez-Piazza}
\address{Departmento de An\'alisis Matem\'atico and IMUS, Facultad de Matem\'aticas, Universidad
	de Sevilla, Calle Tarfia, s/n 41012 Sevilla, Spain}
\email{piazza@us.es}


\date{\today}


\thanks{This research was supported in part by Ministerio de Econom\'{\i}a y Competitividad, Spain,  and the European Union (FEDER), project PGC2018-094215-13-100,  and Junta de Andaluc{\'i}a, FQM133 and FQM-104.}

\maketitle

\begin{abstract}
In this paper we characterize the boundedness, compactness, and weak compactness of the integration operators
\begin{align*}
T_g (f)(z)=\int_{0}^{z} f(w)g'(w)\ dw
\end{align*}
acting on the average radial integrability spaces $RM(p,q)$. For these purposes, we develop different tools such as a description of the bidual of $RM(p,0)$ and estimates of the norm of these spaces using the derivative of the functions, a family of results that we call Littlewood-Paley type inequalities.
\end{abstract}

\tableofcontents

\section{Introduction}

Let $X$ be a Banach space of analytic functions on the unit disc $\D$. Let $T_g$ be given by 
\begin{align*}
T_{g}(f)(z)=\int_{0}^{z} f(w)g'(w)\ dw,\quad f\in X,
\end{align*}
where $g:\D\rightarrow \C$ is an analytic function. This operator is called integration operator.

 In \cite{pommerenke_schlichte_1977} Ch. Pommerenke proved that this operator is bounded  on the Hardy space $H^2$ if and only if $g$ belongs to the space of analytic functions of bounded mean oscillation $BMOA$. Moreover, he also obtained that $T_g$ is compact if and only if $g$ belongs to the space of analytic functions of vanishing mean oscillation $VMOA$. Later on, A. Aleman and A. G. Siskakis extended these results \cite{aleman_integral_1995} to Hardy spaces $H^{p}$ with $1\leq p<+\infty$. 
In addition,  they also showed in \cite{aleman_integration_1997} that the integration operator $T_g$ is bounded on the Bergman space $A^p$, with $1\leq p<+\infty$, if and only if $g$ belongs to the Bloch space $\mathcal{B}$. Additionally, they obtained a similar result for the compactness of $T_g$ but in this case the function $g$ must belong to the little Bloch space $\mathcal{B}_0$. Let us point out that this characterization is still an open problem for $H^\infty$.

In this article we are going to study this kind of results for a family of spaces that encompasses the Bergman spaces $A^p$ and the Hardy spaces $H^p$. This family is the one formed by the space of average radial integrability, that we denote by $RM(p,q)$ (see Definition~\ref{definitionRM}).
In Section 5 we provide a characterization of when the integration operator $T_g$ is bounded or compact over the average radial integrability spaces $RM(p,q)$ (see Theorem~\ref{thboundedopinteg}). For $p<+\infty$, we show that $T_g$ is bounded (resp. compact) over the $RM(p,q)$ spaces if and only if $g$ belongs to the Bloch space $\mathcal{B}$ (resp. little Bloch $\mathcal{B}_0$).

The main issue in this paper, that will be addressed in Section 6, is the weak compactness of the integration operator $T_g$ acting on $RM(p,q)$. When $1<p\leq+\infty$, $1<q<+\infty$ the space $RM(p,q)$ is reflexive. Therefore, the interesting cases are when either $p$ or $q$ belongs to $\{1,+\infty\}$. In this setting we prove:

\begin{theorem} \label{Thm:introduccion} Let $g\in\mathcal{B}$. Then
	\begin{enumerate}
		\item If $1<p< +\infty$, the operator $T_g: RM(p,\infty)\rightarrow RM(p,\infty)$ is weakly compact if and only if $T_g(RM(p,\infty))\subset RM(p,0)$ (see Theorem \ref{weakcompactnessRM-p-infty}).
		\item If $1<q< +\infty$, the operator $T_g: RM(1,q)\rightarrow RM(1,q)$ is weakly compact if and only if $T_g:RM(1,q)\rightarrow RM(1,1)$ is compact and if and only if $g$ belongs to the weakly little Bloch space (see Definition \ref{def:weaklylittleBloch}, Proposition \ref{Prop: weak-compact-1-q} and Theorem \ref{Thm:weakly-compact-RM(1,q)}).
		\item If $1\leq p< +\infty$, the operator $T_g: RM(p,1)\rightarrow RM(p,1)$ is weakly compact if and only if $T_g:RM(p,1)\rightarrow RM(p,1)$ is compact and if and only if $g$ belongs to the little Bloch space (see Corollary \ref{Cor:weak-compactness-p1}).	
	\end{enumerate}
\end{theorem}

In order to obtain these results it has been necessary a study of the bidual of the space $RM(p,0)$ in Section 4. The analogous result to the statement (1) in above theorem for $p=+\infty$ had already been proved in \cite{CPPR} by the second author, J.A. Pel\'aez, Ch. Pommerenke, and J. R\"atty\"a. They obtained that the operator $T_g: H^{\infty}\rightarrow H^{\infty}$ is weakly compact if and only if $T_g: H^{\infty}\rightarrow A$ is bounded, where $A$ is the disc algebra, that turns out to be the closure of polynomials in $H^\infty$. 

Statement (2) in Theorem \ref{Thm:introduccion} relies  on a characterization of the weak compactness of a certain operator into $L^{1}([0,1]\times \T)$ that depend on a  classical result of C. Fefferman and E. Stein about the maximal function on $\ell^{p}$-valued functions.

We also point out that integration operators which do not satisfy statement (2)  in Theorem \ref{Thm:introduccion} are also characterized by the property of  fixing a copy of $\ell^{1}$ (see Proposition \ref{Prop: weak-compact-1-q}) whereas operators not satisfying statement (3)  are those whose adjoints fix a copy of $c_{0}$ (see Corollary \ref{Cor:T*fixes c0}). This type of $\ell^{p}$-singularity has been studied in the setting of integration operators on Hardy spaces (see, e.g., \cite{MNST}).

Another important tool for the study of the integration operator, that we present in Section 3, is Littlewood-Paley type inequalities, that is, for $1\leq p< +\infty$, $1\leq q\leq +\infty$ we obtained the estimate 
$$\rho_{p,q}(f)\leq p\cdot\rho_{p,q}(f'(z)(1-|z|))+|f(0)|$$
for $f\in RM(p,q)$ (Proposition~\ref{bergtypine}). For $p=+\infty$, these inequalities are not satisfied. In Subsection \ref{Subsection Converse}, we also prove that a converse inequality holds for the cases  $1< p, q<+\infty$, $(1,q)$ with $1\leq q<+\infty$, and $(\infty,q)$ with $1\leq q\leq+\infty$.

Throughout the paper the letter $C=C(\cdot)$ will denote an absolute constant whose value depends on the parameters indicated in the parenthesis, and may change from one occurrence to another. We will use the notation $a\lesssim b$ if there exists a constant $C=C(\cdot)>0$ such that $a\leq C b$, and $a\gtrsim b$ is understood in an analogous manner. In particular, if $a\lesssim b$ and $a\gtrsim b$, then we will write $a\asymp b$.
If $X$ is a Banach space, we denote its closed unit ball by $B_X$. 
Finally, 
if $p\in[1,+\infty]$, we define the conjugated index $p'$ such that $\frac{1}{p}+\frac{1}{p'}=1$. 

\noindent {\bf Acknowledgments.} We are grateful to Professor Jos\'e \'Angel Pel\'aez   for calling our attention to the paper \cite{Pavlovic} and to Professor  Daniel Girela for some useful comments. 

\section{Definition and first properties}

In this section we recall the definition of the spaces of average radial integrabilty introduced in \cite{Aguilar-Contreras-Piazza} and summarize their main properties for the sake of being self-contained. 
\begin{definition}\label{definitionRM}
	Let $0< p,q \leq +\infty$. We define the spaces of analytic functions
	$$
		RM(p,q)=\{f\in\mathcal{H}(\D)\ :\rho_{p,q}(f)<+\infty\}
	$$
	where
	\begin{equation*}
	\begin{split}
	\rho_ {p,q}(f)&=\left(\frac{1}{2\pi}\int_{0}^{2\pi} \left(\int_{0}^{1} |f(r e^{i t})|^p \ dr \right)^{q/p}dt\right)^{1/q}, \quad \text{ if } p,q<+\infty,\\
	\rho_ {p,\infty}(f)&=\esssup_{t\in[0,2\pi)}\left(\int_{0}^{1} |f(r e^{i t})|^p \ dr \right)^{1/p}, \quad \text{ if } p<+\infty, \\
	\rho_ {\infty,q}(f)&=\left(\frac{1}{2\pi}\int_{0}^{2\pi} \left(\sup_{r\in [0,1)} |f(r e^{i t})| \right)^{q}dt\right)^{1/q},\quad\text{ if } q<+\infty,\\
	\rho_{\infty,\infty}(f)&=\|f\|_{H^{\infty}}=\sup_{z\in\D}|f(z)|.
	\end{split}
	\end{equation*}	
\end{definition}

In the definition of $\rho_{p,\infty}$, the essential supremum can be replaces by the supremum (see \cite[Remark 2.2]{Aguilar-Contreras-Piazza}). It is worth recalling that $RM(p,q)$ endowed with the norm $\rho_{p,q}$ is a Banach space whenever $1\leq p,q\leq +\infty$.
For certain parameters $p,q$ these spaces $RM(p,q)$ are well known spaces. Namely, it is clear that $RM(p,p)$ is nothing but the Bergman space $A^p$, for $1\leq p<+\infty$. In addition, for $1\leq q\leq +\infty$, one can check that $RM(\infty, q) =H^q$.
Another interesting space that fits in this family is the space of bounded radial variation $BRV$, that is the space of analytic functions such that $f'\in RM(1,\infty)$. 

\begin{remark}
	We will also use the notation $\rho_{p,q}(f)$ for measurable functions $f:\D\mapsto \C$, replacing $\sup$ by $\esssup$ in the above definition. For  $1\leq p,q\leq +\infty$ we will denote by $L^q(\mathbb{T},L^p[0,1])$ the spaces of the measurable functions on $\D$ such that $\rho_{p,q}$-norm is finite.
\end{remark}

\begin{definition}
	Let $1\leq p\leq+\infty$. We define the subspace $RM(p,0)$ of $RM(p,\infty)$
	\begin{align*}
	RM(p,0):=\left\{ f\in \mathcal{H}(\D):\lim\limits_{\rho\rightarrow 1} \sup_{\theta} \left(\int_{\rho}^{1} |f(re^{i\theta})|^{p} dr\right)^{1/p}=0\right\}.
	\end{align*} 
\end{definition}

Given a holomorphic function $f$ in the unit disc and $0<r<1$, we define $f_{r}(z):=f(rz)$, for all $z\in \D$.

\begin{proposition}\label{main-properties}
	Let $1\leq p,q \leq+\infty$.  
	\begin{enumerate}
\item \cite[Proposition 2.8]{Aguilar-Contreras-Piazza} If $z\in \D$, then the functional $\delta_z:RM(p,q)\to \C$ given by $\delta_z(f):=f(z)$, for all $f\in RM(p,q)$, is continuous and 
	\begin{align*}
	\| \delta_{z}\|_{(RM(p,q))^{\ast}} \asymp \frac{1}{(1-|z|)^{\frac{1}{p}+\frac{1}{q}}}.
	\end{align*}
\item \cite[Proposition 2.9]{Aguilar-Contreras-Piazza} If $z\in \D$, then the functional $\delta'_z:RM(p,q)\to \C$ given by $\delta_z(f):=f'(z)$, for all $f\in RM(p,q)$, is continuous, 
	\begin{align*}
	\| \delta'_{z}\|_{(RM(p,q))^{\ast}}\asymp \frac{1}{(1-|z|)^{\frac{1}{p}+\frac{1}{q}+1}}\quad \textrm{and hence }\quad \| \delta'_{z}\|_{(RM(p,q))^{\ast}}\asymp \frac{\| \delta_{z}\|_{(RM(p,q))^{\ast}}}{1-|z|}.
	\end{align*}
\item \cite[Corollary 2.17]{Aguilar-Contreras-Piazza} If $z\in\D$, then 
	\begin{align*}
	\|\delta_{z}\|_{(RM(p,0))^{\ast}}=\|\delta_{z}\|_{(RM(p,\infty))^\ast}\quad \mbox{and } \quad \|\delta'_{z}\|_{(RM(p,0))^{\ast}}=\|\delta'_{z}\|_{(RM(p,\infty))^\ast}.
	\end{align*}
\item \cite[Proposition 2.12 and Proposition 2.13]{Aguilar-Contreras-Piazza} Let $q<+\infty$.	If $f\in RM(p,q)$, then $\rho_{p,q}(f-f_{r})\rightarrow 0$ when $r\rightarrow 1^{-}$. Thus, polynomials are dense in $RM(p,q)$.
\item \cite[Proposition 2.15 and Corollary 2.16]{Aguilar-Contreras-Piazza} Let $p <+\infty$ and  $f\in RM(p,\infty)$. Then, $f\in RM(p,0)$ if and only if 
\begin{equation}\label{Eq:propp0unif}
\rho_{p,\infty}(f-f_{r})\rightarrow 0
\end{equation} when $r\rightarrow 1$.
Thus, polynomials are dense in $RM(p,0)$.
\item \cite[Proposition 3.7]{Aguilar-Contreras-Piazza} Let $f\in RM(p,\infty)$ and $\sigma\in \partial \D$. Then, for the non-tangential limit we have $\angle\lim\limits_{z\rightarrow \sigma}  f(z)(1-\overline \sigma z)^{1/p}=0$.
\item The convergence of sequences in the norm topology, implies the convergence uniformly on compact subsets of the unit disc.
	\end{enumerate}
\end{proposition}

\begin{proposition}\cite[Proposition 2.5]{Aguilar-Contreras-Piazza}\label{lacunary}
Let $\{n_k\}_{k=0}^{\infty}$ be a lacunary sequence (that is, \linebreak[4] $\inf_k\frac{n_{k+1}}{n_k}>1$) such that $n_{k}\in \N\setminus\{0\}$, $1\leq p<+\infty$ and $1\leq q \leq+ \infty$. Then 
$f(z)=\sum_{k=0}^{\infty} \alpha_{k} z^{n_k}$
belongs to $RM(p,q)$ if and only if
\begin{align*}
\sum_{k=0}^{\infty} \frac{|\alpha_k|^p}{n_k}<+\infty.
\end{align*}
 Moreover, it is satisfied that 
\begin{align*}
\rho_{p,q}(f)\asymp \left(\sum_{k=0}^{\infty} \frac{|a_k|^p}{n_k}\right)^{1/p}.
\end{align*}
\end{proposition}

%

\begin{definition}
	The Bloch space $\mathcal{B}$ is the space of analytic functions $f$ such that
	\begin{align*}
	\sup_{z\in\D} (1-|z|^2)|f'(z)|<+\infty
	\end{align*}
	and the little Bloch space, denoted by $\mathcal{B}_0$, is the closed subspace of $\mathcal{B}$ consisting of analytic functions $f$ with 
	\begin{align*}
	\lim_{|z|\to 1}(1-|z|^2)|f'(z)|=0.
	\end{align*}
	The Bloch space $\mathcal{B}$ is a Banach space with norm $\|f\|_{\mathcal{B}}=|f(0)|+\sup_{z\in\D} (1-|z|^2)|f'(z)|$ and the little Bloch $\mathcal{B}_0$ coincides with the closure of polynomials in $\mathcal{B}$.
\end{definition}

\section{Littlewood-Paley type inequalities}

In this section we show the Littlewood-Paley type inequalities associated to $RM(p,q)$ and their converses for certain cases. 

\subsection{Littlewood-Paley type inequalities}

These inequalities will have a great importance in the proof of the fact that the belonging of $g$ to the Bloch space is a sufficient condition for the boundedness of the integration operator $T_g$ acting on $RM(p,q)$.

\begin{lemma}\label{lemmaLPtypeineq}
	Let $1\leq p<+\infty$. If $f\in \mathcal{C}^{1}([0,1))\cap L^p([0,1])$, then it satisfies that
	\begin{align*}
	\left(\int_{0}^{1}|f(x)|^p\ dx\right)^{1/p}\leq p \left(\int_{0}^{1} |f'(x)|^p(1-x)^p\ dx\right)^{1/p}+|f(0)|.
	\end{align*}
\end{lemma}
\begin{proof}
Without lost of generality we can assume that $f(0)=0$. To prove this inequality we need study first the case $p=1$. We have that
\begin{align*}
\int_{0}^{1}|f(x)|\ dx=\int_{0}^{1}\left|\int_{0}^{x} f'(t)\ dt\right|\ dx\leq \int_{0}^{1}\int_{0}^{x} |f'(t)|\ dt\ dx
\end{align*}
Applying Fubini's theorem it follows that
\begin{align*}
\int_{0}^{1}|f(x)|\ dx\leq \int_{0}^{1} |f'(t)|\left(\int_{t}^{1} \ dx\right)\ dt=\int_{0}^{1} |f'(t)| (1-t)\, dt.
\end{align*}
To prove such inequality for $p>1$, we prove that the operator 
\begin{align*}
g\mapsto Rg(x)=\int_{0}^{x} \frac{g(t)}{1-t}\ dt
\end{align*}
is bounded in $L^p([0,1])$. Notice that $\|R{g}\|_{L^p([0,1])}=\sup_{h\in B_{L^{p'}([0,1])}}{|<Rg,h>|}$ and
\begin{align*}
<Rg,h>&=\int_{0}^{1}\left(\int_{0}^{x}\frac{g(t)}{1-t}\ dt\right)h(x)\ dx =\int_{0}^{1}g(t) \left(\frac{1}{1-t}\int_{t}^{1}h(x)\ dx\right)\ dt.
\end{align*}
Let $Ch(t):=\frac{1}{1-t}\int_{t}^{1}h(x)\ dx$ for an integrable function $h$. Using \cite[Exercise 14, p. 72]{rudin_real_1987} with the function $h(1-t)\chi_{[0,1]}(t)$, $t\in (0,\infty)$, we have that
\begin{align*}
\left(\int_{0}^{1}\left|\frac{1}{x}\int_{0}^{x} h(1-t)\, dt\right|^{p'}\,dx \right)^{1/p'}\leq \left(\int_{0}^{\infty}\left|\frac{1}{x}\int_{0}^{x} h(1-t)\chi_{[0,1]}(t) \, dt\right|^{p'}  \, dx\right)^{1/p'}\leq p \|h\|_{L^{p'}}.
\end{align*}

Doing a change of variable, we have that $\|Ch\|_{L^{p'}([0,1])}\leq p \|h \|_{L^{p'}([0,1])}$ and therefore, by duality, $||Rg ||_{L^{p}([0,1])} \leq p ||g ||_{L^{p}([0,1])}$. Taking $g(x)=f'(x)(1-x)$, $x\in [0,1]$, we conclude
$$\|f\|_{L^{p}([0,1])}\leq p \|f'(x)(1-x)\|_{L^{p}([0,1])}.$$
\end{proof}
\begin{proposition}\label{bergtypine}
	Let $1\leq p <+\infty$, $1\leq q\leq +\infty$. If $f\in RM(p,q)$ then we have
	\begin{align}\label{ineqlittlepaleyineq}
	\rho_{p,q}(f)\leq p\ \rho_{p,q}(f'(z)(1-|z|))+|f(0)|.
	\end{align}
\end{proposition}

\begin{proof}
The result follows using Lemma~\ref{lemmaLPtypeineq} and taking the $L^q([0,2\pi])$-norm.
\end{proof}

\begin{remark}Notice that Littlewood-Paley type inequalities are not true for the cases where $p=\infty$, that is, for Hardy spaces. This can be seen through lacunary series.	Assume that $\rho_{\infty,q}(f)\lesssim\ \rho_{\infty,q}(f'(z)(1-|z|))+|f(0)|$. Taking the lacunary series
	\begin{align*}
	f(z)=\sum_{k=0}^{\infty} z^{2^k}
	\end{align*}
	we would have 
	\begin{align*}
	\|f\|_{H^{q}} \asymp \rho_{\infty,q}(f)\lesssim \rho_{\infty,q}(f'(z)(1-|z|))\leq \|f\|_{\mathcal{B}}.
	\end{align*}
	But, this is impossible since $f$ belongs to the Bloch space $\mathcal{B}$ but not to Hardy space $H^q$ (see \cite[p. 241]{duren_theory_2000}).  \newline
\end{remark}

Next, we will show the version of the Littlewood-Paley type inequalities for the spaces $RM(p,0)$. We need a preliminary result that will be used to show a sufficient condition on the functions of $RM(p,\infty)$ to belong to the subspace $RM(p,0)$. 

\begin{proposition}\label{berginrp0}
	Let $1\leq p <+\infty$. If $f\in\mathcal{H}(\D)$ we have
	\begin{align*}
	\left(\int_{\rho}^{1} |f(re^{i\theta})|^{p}\ dr \right)^{1/p}&\leq p \left(\int_{\rho}^{1} |f'(re^{i\theta})|^{p}(1-r)^{p}\ dr \right)^{1/p} +  (1-\rho)^{1/p}|f(\rho e^{i\theta})|
	\end{align*}
	for all $\rho\in [0,1)$ and all $\theta\in [0,2\pi]$.
\end{proposition}

\begin{proof}Fix $\rho\in [0,1)$.
	Consider the function $g(r)=f((r+(1-r)\rho)e^{i\theta})$ and apply Lemma~\ref{lemmaLPtypeineq},
	\begin{align*}
	&\left(\int_{0}^{1}|f((r+(1-r)\rho)e^{i\theta})|^{p}\ dr\right)^{1/p} \\
	&\leq  p \left(\int_{0}^{1} |f'((r+(1-r)\rho)e^{i\theta})|^{p} (1-\rho)^{p}(1-r)^{p}\ dr\right)^{1/p}+|f(\rho e^{i\theta})|.
	\end{align*}
	Using the change of variable $u=r+(1-r)\rho$ we obtain 
	\begin{align*}
\left(\int_{\rho}^{1}|f(ue^{i\theta})|^{p} \ \frac{du}{1-\rho}\right)^{1/p} \leq  p\left(\int_{\rho}^{1} |f'(u e^{i\theta})|^{p} (1-u)^{p}\ \frac{du}{1-\rho}\right)^{1/p}+|f(\rho e^{i\theta})|.
	\end{align*}
	Therefore, we have concluded the proof.
\end{proof}


\begin{proposition}\label{rhoLPEQ}
	Let $1\leq p <+\infty$. If $f\in \mathcal{H}(\D)$ and satisfies
	\begin{align}\label{dernboundsmall}
	\lim\limits_{\rho\rightarrow 1^{-}}\sup_{\theta \in [0,2\pi]} \left(\int_{\rho}^{1}|f'(re^{i\theta})|^{p} (1-r)^{p}\ dr\right)^{1/p}=0
	\end{align}
	then $\lim\limits_{\rho\rightarrow 1^{-}} \sup_{\theta}\ (1-\rho)^{1/p}|f(\rho e^{i\theta})|=0$ and $f\in RM(p,0)$.
\end{proposition}
\begin{proof}
	If $f\in \mathcal{H}(\D)$ and satisfies \eqref{dernboundsmall}, then we clearly obtain that
	\begin{align*}
\sup_{\theta \in [0,2\pi]} \left(\int_{0}^{1}|f'(re^{i\theta})|^{p} (1-r)^{p}\ dr\right)^{1/p}<+\infty.
	\end{align*}
	Using Proposition~\ref{bergtypine} it follows that $\rho_{p,\infty}(f)<+\infty$, that is, $f\in RM(p,\infty)$.  
	So, without loss of generality we assume that $f(0)=0$ and $\rho_{p,\infty}(f)=1$. We observe that
	\begin{align}\label{desp0}
	(1-\rho)^{1/p}|f(\rho e^{i\theta})|=(1-\rho)^{1/p}\left|\int_{0}^{\rho} f'(te^{i\theta})e^{i\theta}\ dt\right|\leq (1-\rho)^{1/p} \int_{0}^{\rho} |f'(te^{i\theta})|\ dt.
	\end{align}
	
	Let us see that $\lim\limits_{\rho\rightarrow 1^{-}} \sup_{\theta}  (1-\rho)^{1/p} \int_{0}^{\rho} |f'(te^{i\theta})|\, dt=0$. Fix $0<\rho_1<\rho$. 
	Using that $|f'(z)|\leq \frac{C}{(1-|z|)^{1+\frac{1}{p}}}$ for all $z$ and for some constant $C=C(p)>0$ (Proposition~\ref{main-properties}(2)) we obtain that
	\begin{align*}
	\int_{0}^{\rho}|f'(te^{i\theta})|(1-\rho)^{1/p}\ dt &\leq C \frac{(1-\rho)^{1/p}}{(1-\rho_1)^{1/p}}\int_{0}^{\rho_1}\frac{(1-\rho_1)^{1/p}}{(1-t)^{1+\frac{1}{p}}}\ dt+\int_{\rho_1}^{\rho} |f'(te^{i\theta})|(1-\rho)^{1/p}\ dt\\
	&\leq C p\frac{(1-\rho)^{1/p}}{(1-\rho_1)^{1/p}}+\int_{\rho_1}^{\rho} |f'(te^{i\theta})|(1-t)\frac{(1-\rho)^{1/p}}{(1-t)}\ dt.
	\end{align*}
	
	If $p=1$
	\begin{align*}
	\int_{0}^{\rho}|f'(te^{i\theta})|(1-\rho)\ dt&\leq C  \frac{1-\rho}{1-\rho_1}+\int_{\rho_1}^{\rho} |f'(te^{i\theta})| (1-t)\ dt\\
	&\leq C  \frac{1-\rho}{1-\rho_1}+\int_{\rho_1}^{1} |f'(te^{i\theta})| (1-t)\ dt.
	\end{align*}
	Using Hölder's inequality for $1<p<+\infty$, we obtain
	\begin{align*}
	&\int_{0}^{\rho}|f'(te^{i\theta})|(1-\rho)^{1/p}\ dt\\
	&\leq Cp \frac{(1-\rho)^{1/p}}{(1-\rho_1)^{1/p}}+ \left( \int_{\rho_{1}}^{\rho}|f'(te^{i\theta})|^{p}(1-t)^{p}\ dt\right)^{1/p}\left(\int_{\rho_1}^{\rho} \frac{(1-\rho)^{p'/p}}{(1-t)^{p'}}\ dt\right)^{1/p'}\\
	&\leq  Cp \frac{(1-\rho)^{1/p}}{(1-\rho_1)^{1/p}}+\left( \int_{\rho_{1}}^{1}|f'(te^{i\theta})|^{p}(1-t)^{p}\ dt\right)^{1/p}\frac{1}{(p'-1)^{1/p'}}.
	\end{align*}
	
	Now, taking supremum with respect to $\theta$ in the previous inequalities and considering $\rho_1=1-\sqrt{1-\rho}$, we obtain that 
	\begin{align*}
	\lim\limits_{\rho\rightarrow 1^{-}} \sup_{\theta} \int_{0}^{\rho}|f'(te^{i\theta})|(1-\rho)^{1/p}\ dt=0.
	\end{align*}
	Therefore, bearing in mind \eqref{desp0} it follows that $\lim\limits_{\rho\rightarrow 1^{-}} \sup_{\theta} \ (1-\rho)^{1/p}|f(\rho e^{i\theta})|=0$ and by means of Proposition~\ref{berginrp0} we conclude that $f\in RM(p,0)$.
\end{proof}


\subsection{Converse Littlewood-Paley type inequalities} \label{Subsection Converse}
Now, we tackle the converse Littlewood-Paley inequality for the cases $1< p, q<+\infty$, $(1,q)$ with $1\leq q<+\infty$, and $(\infty,q)$ with $1\leq q\leq+\infty$.
Firstly, we recall the Luecking regions and their expanded regions.

\begin{definition}
	Given a non-negative integer $n$, set
	\begin{align*}
	\Gamma_{n}=\left\{z\in\D\ :\ 1-\frac{1}{2^{n}}\leq |z|<1-\frac{1}{2^{n+1}}\right\}.
	\end{align*}
	We define the \emph{Luecking regions} $R_{n,j}$ as follows
	\begin{align*}
	R_{n,j}=\left\{z\in\Gamma_{n}\ :\ \arg(z)\in\left[\frac{2\pi j}{2^{n}},\frac{2\pi (j+1)}{2^{n}}\right)\right\},\quad j=0,1,\dots, 2^{n}-1.
	\end{align*}

	
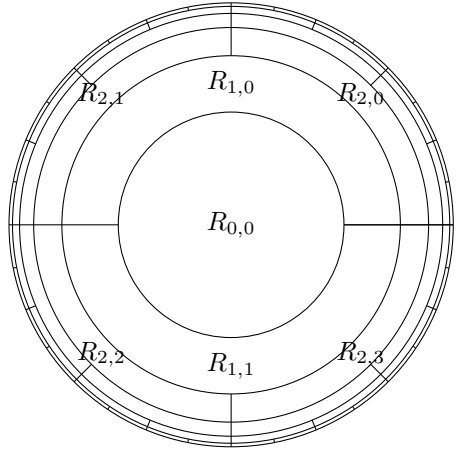
\begin{center}
	\begin{tikzpicture}[scale=3]
	\draw (0,0) circle (1/2);
	\draw (0,0) node[font=\fontsize{10}{0}\selectfont] {$R_{0,0}$};
	\foreach \k in {2,...,6}
		{
		\pgfmathsetmacro{\b}{2^-\k};
		\pgfmathsetmacro{\c}{2^\k/2};
		\draw(0,0) circle (1-\b);
		\foreach \m in {0,...,\c}
			{
			\draw (360.*\m*\b*2:{1-2*\b}) -- (360.*\m*\b*2:{1-\b});
			}
		}
	\foreach \k in {2,...,3}
	{
		\pgfmathsetmacro{\b}{2^-\k};
		\pgfmathsetmacro{\c}{2^\k/2-1};
		\foreach \m in {0,...,\c}
		{
			\pgfmathsetmacro{\x}{int(\k-1)};
			\pgfmathsetmacro{\y}{\m};			
			\draw (360.*\b+360.*\m*\b*2:{1-3*\b/2}) node[font=\fontsize{10}{0}\selectfont] {$R_{\x,\y}$};
		}
	}
	\end{tikzpicture}
	\end{center}

	The \emph{expanded Luecking region} $\tilde{R}_{n,j}$ is the union of $R_{n,j}$ and all regions contiguous with it. That is,  $$\tilde{R}_{n,j}= \bigcup_{\partial R_{n,j}\cap \partial R_{m,k}\neq \emptyset} R_{m,k}.$$

\end{definition}

Now, we state some properties of these regions that will be useful for the estimation of the norm of certain maximal operators over these regions. The proof of these facts will be left to the reader.

\begin{lemma}\label{lemmaNC}
	Let $z\in R_{n,j}$, then $D\left(z,\frac{1-|z|}{2}\right)\subset \tilde{R}_{n,j}$.
\end{lemma}

\begin{lemma}\label{remarkNC}
	Denote by $NC(R_{n,j})$ the number of regions $R_{m,k}$ such that $\partial R_{n,j}\cap \partial R_{m,k}\neq \emptyset$. Then,
	\begin{itemize}
		\item $NC(R_{0,0})=3$;
		\item $NC(R_{1,j})=7$, $j=0,1$;
		\item If $n\geq 2$, then $NC(R_{n,j})=9$, $j=0,\dots, 2^{n}-1$.
		\end{itemize}
	In particular, $NC(R_{n,j})\leq 9$.
\end{lemma}

\begin{lemma}\label{areaRnj}
	Let $n$ be a non-negative integer. Then 
	$$m_{2}(R_{n,j})\asymp m_{2}(\tilde{{R}}_{n,j})\asymp 4^{-n}$$
	 for $j=0,1,\dots, 2^n-1$.
\end{lemma}

From now on, given a measurable set $A\subset \D$ of positive measure we will denote the mean $\frac{1}{m_2(A)}\int_{A} f\ dm_2$ by $\fint_{A} f\ dm_2$  for $f\in L^1(A)$.

\begin{definition}
	For every locally integrable  function $f$ on $\D$ we set:
	\begin{enumerate}
		\item $M_{R}f:=\sum_{n,j}\left(\fint_{R_{n,j}} |f|\ dm_{2}\right)\chi_{R_{n,j}}.$
		\item $M_{\tilde{{R}}}f:=\sum_{n,j} \left(\fint_{\tilde{R}_{n,j}} |f|\ dm_{2}\right)\chi_{R_{n,j}}.$
		\item $M_{D}f(z):=\fint_{D\left(z,\frac{1-|z|}{2}\right)} |f|\ dm_2, $ for all $z\in \D$.
	\end{enumerate} 
\end{definition}
Now, we estimate these sublinear operators in the following proposition.

\begin{proposition}\label{boundmaximaloperators} Let $f$ be a locally integrable function on $\D$.
\begin{enumerate}
	\item If $1\leq p\leq q<+\infty$, there is a constant $C(p,q)>0$ such that $\rho_{p,q}(M_{R}f)\ \leq C(p,q) \rho_{p,q}(f)$.\newline
	\item If $1\leq p\leq q<+\infty$ or $1< q\leq p\leq+\infty$, there is a constant $C(p,q)>0$ such that 
	$\rho_{p,q}(M_{{\tilde{R}}}f)\leq C(p,q)\rho_{p,q}(f)$.\newline
	\item If $1\leq p\leq q<+\infty$ or $1<q\leq p	\leq  +\infty$, there is a constant $C(p,q)>0$ such that
	$\rho_{p,q}(M_{D}f)\leq C(p,q)\rho_{p,q}(f)$.
\end{enumerate}
\end{proposition}
\begin{proof}
(1) Since $1\leq p\leq q<\infty$ we have that  $r=\frac{q}{p}\geq 1$. So
\begin{align*}
\rho_{p,q}^{p}(M_{R}f)&=\left(\int_{0}^{2\pi} \left(\int_{0}^{1} |M_{R}f(re^{i\theta})|^{p}\ dr\right)^{r}\ \frac{d\theta}{2\pi}\right)^{1/r}\\
&=\sup_{\substack{\xi\in B_{L^{r'}(\mathbb{T})} \\ \xi\geq 0}} \int_{0}^{2\pi} \left(\int_{0}^{1} |M_{R}f(re^{i\theta})|^p\ dr\right)\xi(e^{i\theta})\frac{d\theta}{2\pi}.
\end{align*}

Fix $\xi\in B_{L^{r'}(\mathbb{T})}$ with $\xi\geq 0$. Bearing in mind that the regions $R_{n,j}$ are pairwise disjoint we have that
\begin{align*}
\int_{0}^{2\pi} \left(\int_{0}^{1} |M_{R}f(re^{i\theta})|^p\ dr\right)\xi(e^{i\theta})\frac{d\theta}{2\pi}\asymp\sum_{n,j} \int_{R_{n,j}} \left(M_{R}f(z)\right)^p \xi\left(\frac{z}{|z|}\right)\ dm_2(z).
\end{align*}
Now, using Jensen's inequality and the fact that $M_{R}f$ is constant in each region $R_{n,j}$, we obtain that 
\begin{align*}
&\int_{0}^{2\pi} \left(\int_{0}^{1} |M_{R}f(re^{i\theta})|^p\ dr\right)\xi(e^{i\theta})\frac{d\theta}{2\pi}\\
&\lesssim \sum_{n,j}\left(\fint_{R_{n,j}} |f(z)|^p\ dm_{2}(z)\right)\int_{R_{n,j}}\xi\left(\frac{z}{|z|}\right)\ dm_{2}(z)\\
&\leq \frac{1}{4\pi}\sum_{n,j}\left(\fint_{R_{n,j}} |f(z)|^p\ dm_{2}(z)\right) m_{1}(I_{n,j})\int_{I_{n,j}}\xi\left(e^{i\theta}\right)\ dm_{1}(\theta)
\end{align*}
where $I_{n,j}=\left\{e^{i\theta}:\ \theta\in \left[\frac{2\pi j}{2^{n}},\frac{2\pi (j+1)}{2^{n}}\right)\right\}$.  Taking into account Lemma~\ref{areaRnj}, observe that $m_1(I_{n,j})\asymp 2^n$ and then $m_2(R_{n,j})\asymp (m_1(I_{n,j}))^2$. Using the Hardy-Littlewood maximal operator $\mathcal{M}$ we have
\begin{align*}
&\int_{0}^{2\pi} \left(\int_{0}^{1} |M_{R}f(re^{i\theta})|^p\ dr\right)\xi(e^{i\theta})\frac{d\theta}{2\pi}\lesssim\sum_{n,j} \int_{R_{n,j}} |f(z)|^{p} \inf_{e^{\theta i}\in I_{n,j}} \mathcal{M} \xi(e^{i\theta})\ dm_{2}(z)\\
&\leq \sum_{n,j} \int_{R_{n,j}} |f(z)|^{p}  \mathcal{M}\xi\left(\frac{z}{|z|}\right)\ dm_{2}(z)=\int_{\D} |f(z)|^{p}  \mathcal{M}\xi\left(\frac{z}{|z|}\right)\ dm_{2}(z)\\
&\lesssim \int_{0}^{2\pi} \mathcal{M}\xi(e^{i\theta}) \left(\int_{0}^{1} |f(re^{i\theta})|^{p}\ dr\right)\ \frac{d\theta}{2 \pi}\leq \|\mathcal{M}(\xi)\|_{L^{r'}(\T)}\ \rho_{p,q}^{q}(f).
\end{align*}

Finally, since $\| \mathcal{M} \xi\|_{L^{r'}(\mathbb{T})} \leq C_{r'} \|\xi\|_{L^{r'}(\mathbb{T})}$ we conclude the proof of statement (1). 

(2) The proof of the case $1\leq p\leq q<+\infty$ follows the same argument we have used in statement (1), but using Lemma~\ref{remarkNC}, Lemma~\ref{areaRnj} and the projection of the regions $\tilde{R}_{n,j}$, $n\in\N$, $j\in\{0,1,\dots,2^n-1\}$ over $\mathbb{T}$ instead of arcs $I_{n,j}$. In fact, it can be shown that the linear operator 
\begin{align*}
\tilde{M}f=\sum_{n,j}\left(\fint_{\tilde{R}_{n,j}}f(z)\ dm_{2}(z)\right)\chi_{R_{n,j}}
\end{align*}
is bounded on the space of measurable functions on $\D$ where the $\rho_{p,q}$-norm is finite. 

Now, we assume that $1<q\leq p\leq+\infty$. Let us describe the adjoint of the operator $\tilde{M}$. Given $g$, we have
\begin{align*}
\int_{\D} \left(\tilde{M}f(z)\right)g(z)\ dm_{2}(z)&=\sum_{n,j}\left(\fint_{\tilde{R}_{n,j}} f(w)\ dm_{2}(w)\right) \left(\int_{{R}_{n,j}} g(z)\ dm_{2}(z)\right)\\
&=\sum_{n,j} \left(\int_{\tilde{R}_{n,j}} f(w)\ dm_{2}(w)\right) \frac{m_{2}(R_{n,j})}{m_{2}(\tilde{R}_{n,j})} \left(\fint_{{R}_{n,j}} g(z)\ dm_{2}(z)\right)\\
&=\int_{\D} f(w)\left(\sum_{n,j} \beta_{n} \left(\fint_{{R}_{n,j}} g(z)\ dm_{2}(z)\right)\chi_{\tilde{R}_{n,j}}(w) \right)\ dm_{2}(w)
\end{align*} 
where $\beta_{n}=\frac{m_{2}(R_{n,j})}{m_{2}(\tilde{R}_{n,j})}$ (notice that this quotient does not depend on $j$). Hence, the adjoint operator is 
\begin{align*}
\tilde{M}^{\ast}f=\sum_{n,j} \beta_{n} \left(\fint_{{R}_{n,j}} f(z)\ dm_{2}(z)\right)\chi_{\tilde{R}_{n,j}}.
\end{align*}
We know that there is a constant $C(p,q)>0$ such that $\rho_{p',q'}(\tilde{M}f)\leq C(p,q)\rho_{p',q'}(f)$. Thus, $\rho_{p,q}(\tilde{M}^{\ast}f)\leq C(p,q)\rho_{p,q}(f)$. Moreover, we claim that $\tilde{M}f\asymp \tilde{M}^{\ast}f$ for positive functions. Therefore, it follows 
$$
\rho_{p,q}(M_{\tilde{{R}}}f)=\rho_{p,q}(\tilde{M}|f|)\asymp \rho_{p,q}(\tilde{M}^{\ast}f)\leq C(p,q)\rho_{p,q}(f).
$$

To proof the claim take a positive function $f$. Bearing in mind Lemma~\ref{areaRnj}, we have 
\begin{align*}
\tilde{M}^{\ast}f&=\sum_{n,j} \beta_{n} \left(\fint_{{R}_{n,j}} f(z)\ dm_{2}(z)\right)\chi_{\tilde{R}_{n,j}}\\
&=\sum_{n,j} \beta_{n} \sum_{\partial R_{n,j}\cap\partial R_{m,k}\neq \emptyset} \left(\fint_{R_{n,j}} f(z)\ dm_{2}(z)\right)\chi_{R_{m,k}}\\
&=\sum_{m,k}\sum_{\partial R_{n,j}\cap\partial R_{m,k}\neq \emptyset}\frac{1}{m_{2}(\tilde{R}_{n,j})}\left(\int_{R_{n,j}} f(z)\ dm_{2}(z)\right)\chi_{R_{m,k}}\\
&\asymp \sum_{m,k} \left(\fint_{\tilde{R}_{m,k}}f(z)\ dm_{2}(z)\right)\chi_{R_{m,k}}=\tilde{M}f.
\end{align*}

(3)
Let $z\in \D$  and take $R_{n,j}$ such that $z\in R_{n,j}$. Hence, using Lemma~\ref{lemmaNC}, we have $D\left(z,\frac{1-|z|}{2}\right)\subset \tilde{R}_{n,j}$. Also, it can be proved that $ {m_{2}(\tilde{R}_{n,j})}\asymp m_{2}\left(D\left(z,\frac{1-|z|}{2}\right)\right)$.
Therefore, the result follows because  $M_{D}f(z)\lesssim M_{\tilde{R}}f(z)$ for every $z\in \D$. 
\end{proof}

Using this result we obtain the converse Littlewood-Paley inequality for certain cases. These inequalities will be important in the subsequent study of the weak compactness of the operator $T_g: RM(1,q)\rightarrow RM(1,q)$.

	\begin{proposition}\label{converse-little-paley}
	Assume $(p,q)$ are in one of the following three cases: $1< p, q<+\infty$, $(1,q)$ with $1\leq q<+\infty$, or $(\infty,q)$ with $1\leq q\leq+\infty$.
Then, there is a constant $C=C(p,q)>0$ such that
	\begin{align*}
	\rho_{p,q}(f'(z)(1-|z|))\leq C \rho_{p,q}(f), \qquad f\in RM(p,q).
	\end{align*}
\end{proposition}

\begin{proof}
By means of Cauchy's integral formula over $\partial D(z,r)$ for $0\leq r\leq \frac{1-|z|}{2}$ we have that
\begin{align*}
2\pi r^{2} |f'(z)|\leq r \int_{0}^{2\pi} |f(z+re^{i\theta})|\ d\theta.
\end{align*}
Integrating with respect to $r$ over $0\leq r\leq \frac{1-|z|}{2}$ 
\begin{align*}
\frac{1}{3}(1-|z|)|f'(z)|\leq \fint_{D\left(z,\frac{1-|z|}{2}\right)} |f(\xi)|\ dm_{2}(\xi)=M_{D}f(z).
\end{align*}
Assume that $1\leq p\leq q<+\infty$ or $1<q\leq p	\leq  +\infty$.  Taking $RM(p,q)$-norm and using the statement (3) of Proposition~\ref{boundmaximaloperators} we conclude
$$\rho_{p,q}(f'(z)(1-|z|)\leq 3 \rho_{p,q}(M_{D}f)\leq C_{p,q}\rho_{p,q}(f).$$\newline

Now, we continue with the remaining cases, that is, $p=+\infty$, $1\leq q\leq +\infty$. Fix $e^{i\theta}\in \partial\D$ and $r\in (0,1)$. 
From $(1-r)|f'(re^{i\theta})|\lesssim M_{D}f(re^{i\theta})$ and taking a certain Stolz region $\Gamma\left(e^{i\theta}\right)$ such that $D\left(re^{i\theta},\frac{1-r}{2}\right)\subset \Gamma\left(e^{i\theta}\right)$, it follows that
$$(1-r)|f'(re^{i\theta})|\leq 3\sup_{w\in\Gamma\left(e^{i\theta}\right)} |f(w)|=:3H_{f}(e^{i\theta}).$$
 Moreover, if we take supremum with respect to $r$, it follows that $\sup_{r\in[0,1)} (1-r)|f'(re^{i\theta})|\leq 3H_{f}(e^{i\theta})$.
	
Therefore, taking $L^{q}(\mathbb{T})$-norm  for $q\geq1$ and using \cite[Theorem 3.1, p. 55]{garnett_bounded_2007}, we obtain
\begin{align}\label{LPinfq}
\rho_{\infty,q}((1-|z|)f'(z))\leq \|H_f\|_{L^{q}(\mathbb{T})}\leq C_{q} \|f\|_{H^{q}} \leq C_{q}\rho_{\infty,q}(f).
\end{align}
\end{proof}

To finish this section, we point out that we do not know if the converse Littlewood-Paley inequality holds in the cases $(p,1)$ with $1<p<+\infty$ and $(p,\infty)$ with $1\leq p<+\infty$.

\section{On the bidual of $RM(p,0)$}

In this section, we identify in a natural manner the bidual of $RM(p,0)$ with $RM(p,\infty)$.
It is clear that $RM(p,\infty)$ is a subspace of the Bergman space $A^{p}=RM(p,p)$. Throughout this section, we denote by $I$ the inclusion map from $RM(p,0)$ into $A^p$.
We follow the scheme of the proof of K.-M. Perfekt in \cite{perfekt_2013}, but it is worth mentioning we can not use his results directly.

\begin{lemma}\label{densityBergRMP0} Let $1<p<+\infty$. Then $I^*((A^{p})^{\ast})$ is dense (in the norm topology) in  $(RM(p,0))^{\ast}$.
\end{lemma}
\begin{proof}
	Given $\lambda\in (RM(p,0))^{\ast}$, consider the following family of bounded linear functionals $\lambda_{r}(f):=\lambda(f_r)$, $0<r<1$. We will prove that $\lambda_r\in (A^{p})^{\ast}$ and $\lim_{r\to 1}\Vert I^*(\lambda_r)-\lambda\Vert _{(RM(p,0))^{\ast}} =0$. It is clear that $ I^*(\lambda_r)$ acts as $\lambda _r$ over $RM(p,0)$. So that, as customary, with a slight abuse of notation, we will  write $\lambda _r$ instead of $I^*(\lambda_r)$.
	
	Given $f\in A^p$, we have 	
	\begin{align*}
	|\lambda_{r}(f)|=|\lambda(f_r)|\leq \|\lambda\| \, \rho_{p,\infty}(f_r)\leq \|\lambda\| \sup_{D(0,r)}|f(w)|\lesssim \frac{\|\lambda\|}{(1-r)^{2/p}} \rho_{p,p}(f).
	\end{align*}
	Therefore, $\lambda_{r}\in (A^p)^*$.
	
	Assume that $\lim_{r\to 1}\Vert I^*(\lambda_r)-\lambda\Vert _{(RM(p,0))^{\ast}} $ is not $0$. Then there exists $\epsilon>0$, a sequence $\{r_n\}$ in $(0,1)$, with $r_n\rightarrow 1$, and a sequence $\{h_n\}$ in the unit ball of $RM(p,0)$ such that 
	\begin{align*}
	\epsilon<|\lambda(h_n)-\lambda _{r_n} (h_{n})|=|\lambda(h_n-(h_{n})_{r_n})|\leq \|\lambda\| \, \rho_{p,\infty}(h_n-(h_{n})_{r_n})
	\end{align*}
	for all $n$. Writing  $g_n:=h_n-(h_{n})_{r_n}$, we have a bounded sequence $\{g_n\}$ in $RM(p,0)$ that goes to zero uniformly on compacta of $\D$ and such that $\frac{\varepsilon}{\|\lambda\|}<\rho_{p,\infty}(g_n)<3$ for all $n$.
	
Fix a sequence of positive numbers $\{\varepsilon_k\}\in \ell^{p'}$ and $\rho_1\in (0,1)$. There exists $n_1$  such that $\sup\{|g_{n_1}(z)|:\, |z|\leq \rho_1\}<\varepsilon_{1}$. Since $g_{n_1}\in RM(p,0)$, 
we can choose $\rho_2>\rho_1$ so that 
	\begin{align*}
	\sup_{\theta} \left(\int_{\rho_2}^{1} |g_{n_1}(re^{i\theta})|^{p}\ dr\right)^{1/p}<\varepsilon_{1}.
	\end{align*}
With the same argument, we obtain $n_2$ and $\rho_{3}>\rho_2$ such that $\sup\{|g_{n_2}(z)|:\, |z|\leq \rho_2\}<\varepsilon_{2}$ and 
	\begin{align*}
	\sup_{\theta}& \left(\int_{\rho_3}^{1} |g_{n_2}(re^{i\theta})|^{p}\ dr\right)^{1/p}<\varepsilon_{2}.
	\end{align*}
	By induction, we build a sequence $\{g_{n_k}\}$, an increasing sequence of numbers $\{\rho_k\}$ that converges to $1$, such that 
	$\sup\{|g_{n_k}(z)|:\, |z|\leq \rho_k\}<\varepsilon_{k}$ and 
	\begin{align*}
	\sup_{\theta}& \left(\int_{\rho_{k+1}}^{1} |g_{n_k}(re^{i\theta})|^{p}\ dr\right)^{1/p}<\varepsilon_{k}.
	\end{align*}

	
	Given $\{\alpha_k\}\in\ell^p$, we will see that $\sum_{k=0}^{\infty}\alpha_{k} g_{n_k}\in RM(p,0)$. Since the sequence $\{g_{n_k}\}$ goes to zero uniformly on compacta faster than a sequence in $\ell ^{p'}$, we get that $\sum_{k=0}^{\infty}\alpha_{k} g_{n_k}\in \mathcal{H}(\D)$. Moreover
	\begin{align*}
	&\rho_{p,\infty}\left(\sum_{k=1}^{\infty}\alpha_{k}g_{n_k}\right)\\
	&\quad =\sup_{\theta} \left(\int_{0}^{1}\left|\sum_{k=1}^{\infty} \alpha_{k} g_{n_k}(re^{i\theta})\left(\chi_{[0,\rho_k)}(r)+\chi_{[\rho_k,\rho_{k+1})}(r)+\chi_{[\rho_{k+1},1)}(r)\right)\right|^{p}\ dr\right)^{1/p}\\
	&\quad\leq \sup_{\theta} \left(\int_{0}^{1} \left(\sum_{k=1}^{\infty}|\alpha_k|\varepsilon_{k}\right)^{p}\ dr\right)^{1/p}+\sup_{\theta} \left(\int_{0}^{1}\left(\sum_{k=1}^{\infty}|\alpha_k||g_{n_k}(re^{i\theta})|\chi_{[\rho_k,\rho_{k+1})}(r)\right)^{p}\ dr\right)^{1/p}\\
	&\qquad+\sup_{\theta} \left(\int_{0}^{1}\left(\sum_{k=1}^{\infty}|\alpha_k||g_{n_k}(re^{i\theta})|\chi_{[\rho_{k+1},1)}(r)\right)^{p}\ dr\right)^{1/p}\\
	&\quad \leq \|\{\alpha_{k}\}\|_{\ell^p}\|\{\varepsilon_{k}\}\|_{\ell^{p'}}+\sup_{\theta} \left(\sum_{k=1}^{\infty} \int_{\rho_k}^{\rho_{k+1}}|\alpha_{k}|^{p}|g_{n_k}(re^{i\theta})|^{p}\ dr\right)^{1/p}\\
	&\qquad+\sum_{k=1}^{\infty}|\alpha_k| \sup_{\theta}\left(\int_{\rho_{n+1}}^{1} |g_{n_k}(re^{i\theta})|^{p}\ dr\right)^{1/p}\\
	&\quad\leq \|\{\alpha_{k}\}\|_{\ell^p}\|\{\varepsilon_{k}\}\|_{\ell^{p'}}+3\|\{\alpha_{k}\}\|_{\ell^p}+\|\{\alpha_{k}\}\|_{\ell^p}\|\{\varepsilon_{k}\}\|_{\ell^{p'}}=\left(3+2\|\{\varepsilon_{k}\}\|_{\ell^{p'}}\right)\|\{\alpha_{k}\}\|_{\ell^p}.
	\end{align*}
Therefore, $\sum_{k=0}^{\infty}\alpha_{k} g_{n_k}\in RM(p,\infty)$ and moreover this series is convergent in $RM(p,\infty)$. Since $\sum_{k=0}^{N}\alpha_{k} g_{n_k}\in RM(p,0)$	for all $N$ and  $RM(p,0)$ is closed in $RM(p,\infty)$, we conclude that  $\sum_{k=0}^{\infty}\alpha_{k} g_{n_k}\in RM(p,0)$.
Therefore, there exists a bounded linear operator $T:\ell^p\rightarrow RM(p,0)$ such that  $T(e_k)=g_{{n_k}}$, for all $k$. 
	
	Now, if we consider the composition of the operators $\lambda$ and $T$ it follows that $\lambda \circ T: \ell^p\rightarrow \C$ is a bounded linear functional, that is, $\lambda \circ T \in (\ell^p)^{\ast}\cong \ell^{p'}$. So, it satisfies that $(\lambda \circ T)(e_{k})\rightarrow 0$, $k\rightarrow \infty$, but this is impossible because $|(\lambda \circ T)(e_k)|>\varepsilon$ for all $k\in \N$. Therefore, $I^{\ast}((A^p)^{\ast})$ is dense in $(RM(p,0))^{\ast}$.
\end{proof}

	\begin{lemma}\label{aprox-in-Bergman}
		For all $f\in RM(p,\infty)$ there exists a sequence $\{f_n\}\subset RM(p,0)$ such that $f_n\rightarrow f$ in $A^p$ and $\limsup_{n}\rho_{p,\infty}(f_n)\leq\rho_{p,\infty}(f)$.
	\end{lemma}
	\begin{proof}
		Fix $f\in RM(p,\infty)$ and for each $r\in (0,1)$ consider $f_r(z):=f(rz)$, $z\in \D$. We have 
		\begin{align*}
		\sup_{\theta} \left(\int_{\rho}^{1}|f_{r}(s e^{i\theta})|^{p}\ ds\right)^{1/p}=\sup_{\theta} \left(\int_{\rho}^{1}|f(rs e^{i\theta})|^{p}\ ds\right)^{1/p}\leq (1-\rho)^{1/p}\sup_{z\in D(0,r)} |f(z)| \rightarrow 0 
		\end{align*}
		when $\rho\rightarrow 1$.
		That is, the function $f_r$ belongs to $ RM(p,0)$ for all $r<1$.
		
Since $RM(p,\infty)\subset A^{p}$, by  Proposition \ref{main-properties}(4),  $\rho_{p,p}(f_r-f)\rightarrow 0$ when $r\rightarrow 1$. Moreover, since $\int_{0}^{1}|f_{r}(se^{i\theta})|^p\ ds<\frac{1}{r} \int_{0}^{1}|f(se^{i\theta})|^p\ ds$, for all $\theta\in[0,2\pi]$, we have that \linebreak[4] $\limsup_{r\rightarrow 1}\rho_{p,\infty}(f_r)\leq \rho_{p,\infty}(f)$.
	\end{proof}

\begin{theorem}\label{DualRMpinfty}Let $1<p<+\infty$ and the inclusion $I:RM(p,0)\rightarrow A^{p}$. Then $I^{\ast\ast}:(RM(p,0))^{\ast\ast}\rightarrow A^p$ is a continuous and injective inclusion. Moreover
	\begin{enumerate}
		\item[(1)] $I^{\ast\ast}((RM(p,0))^{\ast\ast})= RM(p,\infty),$
		\item[(2)] $I^{\ast\ast}: (RM(p,0))^{\ast\ast}\rightarrow RM(p,\infty)$ is an isometry.
	\end{enumerate} 
If $\{x_{n}\}$ is a bounded sequence in $(RM(p,0))^{\ast\ast}$ that converges to $0$ in the weak-$^{*}$ topology, then $\{I^{\ast\ast}(x_{n})\}$ converges to $0$ uniformly on compact subsets of the unit disc.
\end{theorem}
\begin{proof}
Since the set of all polynomials is dense both in $RM(p,0)$ and $A^p$, then $RM(p,0)$ is dense in $A^p$. Furthermore, it follows that $I^{\ast}: (A^p)^{\ast}\rightarrow (RM(p,0))^{\ast}$ is continuous and injective.
	
Since $I^{\ast}((A^{p})^{\ast})$ is dense in $(RM(p,0))^{\ast}$  (see Lemma~\ref{densityBergRMP0}), we obtain that $${I^{\ast\ast}:(RM(p,0))^{\ast\ast}\rightarrow (A^{p})^{\ast\ast}\cong A^p}$$ is continuous and injective. Moreover, it is easy to see that $I^{\ast\ast}$ acts as the identity on $RM(p,0)$. 
	
Let us show that $RM(p,\infty)\subset I^{\ast\ast}((RM(p,0))^{\ast\ast})$. Given $\psi\in RM(p,\infty)$, by Lemma \ref{aprox-in-Bergman}, there exists a sequence $\{\psi_n\}\subset RM(p,0)$ such that $\psi_n\rightarrow \psi$ in $A^p$ and $\limsup_{n}\rho_{p,\infty}(\psi_n)\leq \rho_{p,\infty}(\psi)$. We can define $\hat{\psi}\in (RM(p,0))^{\ast\ast}$ of the following form
	\begin{align*}
	\hat{\psi}(x^{\ast}):=\lim_{n\rightarrow \infty} x^{\ast}(\psi_{n})
	\end{align*}
	for $x^\ast\in (RM(p,0))^{\ast}$. Notice that it is well-defined since $I^{\ast}((A^p)^\ast)$ is dense in $(RM(p,0))^\ast$ and the sequence $\{\psi_n\}$ is well-defined  and bounded over $I^{\ast}((A^p)^\ast)$. It is clear that $\psi_n \rightarrow \hat{\psi}$ in $\sigma((RM(p,0))^{\ast\ast},(RM(p,0))^{\ast})$. Hence,
	$I^{\ast\ast}(\psi_n)\rightarrow I^{\ast\ast}(\hat{\psi})$ in $(A^p,w^\ast)$.
	
	Moreover, as $\psi_n\rightarrow \psi$ on $A^p$ we have that $\psi_n=I^{\ast\ast}(\psi_n)\rightarrow \psi$ in $(A^p,w)$. From this and the reflexivity of $A^p$, we conclude that $I^{\ast\ast}(\hat{\psi})=\psi$. Moreover, 
	\begin{align}\label{ineqpsihat}
\|\hat{\psi}\|_{(RM(p,0))^{\ast\ast}}\leq\limsup_{n } \rho_{p,\infty}(\psi_{n})\leq \rho_{p,\infty}(\psi).
	\end{align}

	We turn our attention to show the inclusion $I^{\ast\ast}((RM(p,0))^{\ast\ast})\subset RM(p,\infty)$. Let $m\in (RM(p,0))^{\ast\ast}$ and $\psi=I^{\ast\ast}(m)\in A^p$. Using \cite[Proposition 4.1, Chapter V]{conway_course_1990} we have that the unit ball of $RM(p,0)$ is $w^\ast$-dense in the unit ball of $(RM(p,0))^{\ast\ast}$. Moreover, the weak$^\ast$-topology of $(RM(p,0))^{\ast\ast}$ is metrizable in the unit ball, because $(RM(p,0))^{\ast}$ is separable due to Lemma~\ref{densityBergRMP0} and the fact that $(A^p)^\ast$ is separable. We choose a sequence $\{\psi_n\}\subset RM(p,0)$ with $\sup_{n}\rho_{p,\infty}(\psi_{n})\leq \|m\|$ such that $\psi_n\xrightarrow{w^\ast} m$. 
	Therefore, it follows that $x^{\ast}(\psi_n)\rightarrow x^{\ast}(\psi)$ for all $x^{\ast}\in (A^p)^{\ast}$. Bearing in mind that  $\delta_z\in (A^p)^{\ast}$ for all $z\in \D$, we have that $\psi_n(z)\rightarrow \psi(z)$ for all $z\in \D$. Using Fatou's lemma we obtain that
	\begin{align}\label{ineqpsi}
	\rho_{p,\infty}(\psi)\leq \liminf_{n} \rho_{p,\infty}(\psi_n)\leq \|m\|_{(RM(p,0))^{\ast}}.
	\end{align}
	From \eqref{ineqpsihat} and \eqref{ineqpsi} we obtain (2).
	
	Moreover, if $\{x_n\}\in (RM(p,0))^{\ast\ast}\subset (A^p)^{\ast\ast}=A^p$ that converges to $0$ in the weak-$^{*}$ topology, then $\{I^{\ast\ast}(x_{n})\}\subset A^p$ converges pointwise to $0$, and this implies that $\{I^{\ast\ast}(x_{n})\}$ converges to $0$ uniformly on compact subsets of the unit disc.
\end{proof}

\section{Boundedness and compactness of the integration operator $T_g$}

In this section we begin to study the main issue of this paper, that is, the operators of the form
\begin{align*}
T_{g}(f)(z)=\int_{0}^{z} f(\zeta) g'(\zeta)\ d\zeta,
\end{align*}
where $g\in \mathcal H(\D)$, in the spaces $RM(p,q)$.

 \begin{lemma}\label{bloch}
 Let  $1\leq p<+\infty$. Then $\mathcal{B}\subset RM(p,0)$.
 \end{lemma}
 \begin{proof}
 Let $g\in\mathcal{B}$. Then there are $M,C>0$ such that
\begin{align}\label{desi1comp}
|g(z)|\leq M \ln\left(\frac{1}{1-|z|}\right)+C, \quad z\in \D, \textrm{\quad\cite[Proposition 1, p. 43]{duren_schuster_2004}}
\end{align}  
and we have that 
\begin{align*}
\sup_{\theta} \left(\int_{\rho}^{1} |g(se^{i\theta})|^{p}\ ds\right)^{1/p}&\leq \left(\int_{\rho}^{1} \left(M \ln\left(\frac{1}{1-s}\right)+C\right)^{p}\ ds\right)^{1/p} \\
&\leq 	M \left(\int_{\rho}^{1} \ln^{p}\left(\frac{1}{1-s}\right)\ ds\right)^{1/p}+ C(1-\rho)^{1/p} \rightarrow 0
\end{align*}
when $\rho \rightarrow 1$, since $\ln^{p}\left(\frac{1}{1-s}\right)$ is integrable.
Therefore, $\mathcal{B}\subset RM(p,0)$.
\end{proof}
 
By Proposition \ref{main-properties}(7), every bounded sequence in $RM(p,q)$ is uniformly bounded on each compact set of the unit disc and then it is a normal family. Thus a standard argument shows that:
 
\begin{lemma}\label{prop2comp}
	Let $1\leq p, q\leq +\infty$  and let $T:RM(p,q)\to X$ be a linear and bounded operator, where $X$ is a Banach space.
	\begin{enumerate}
\item  If for every bounded sequence $\{f_n\}\subset RM(p,q)$ uniformly convergent on compact sets to $0$ it holds that  $||T(f_n)||\rightarrow 0$, then the operator $T$ is compact.
\item Assuming that $T=T_{g}$ for some holomorphic function $g$ and $X=RM(p,q)$, then $T$ is  compact in $RM(p,q)$ if and only if for every bounded sequence $\{f_n\}\subset RM(p,q)$ uniformly convergent on compact sets to $0$  it  holds that  $\rho_{p,q}(T_{g}(f_n))\rightarrow 0$.
\end{enumerate}
\end{lemma}

\noindent\textbf{Boundedness and compactness.} 
It is well-known that, for $1\leq p<+\infty$, the operator $T_g$ is bounded (resp. compact) over the Hardy spaces $H^p$ if and only if $g\in BMOA$ (resp. $VMOA$) \cite{pommerenke_schlichte_1977,aleman_integral_1995}, and the operator $T_g: A^p\rightarrow A^p$ is bounded (resp. compact) if and only if $g$ belongs to the Bloch space $\mathcal{B}$ (resp. the little Bloch space $\mathcal{B}_0$) \cite{aleman_integration_1997}. Next result completes these characterizations to $RM(p,q)$ whenever $(p,q)\neq (+\infty,+\infty)$.

\begin{theorem}\label{thboundedopinteg}
	Let $1\leq p <+\infty$, $1\leq q\leq +\infty$.  Then
\begin{enumerate}
\item The operator $T_g:RM(p,q)\rightarrow RM(p,q)$ is bounded if and only if $g\in \mathcal{B}$.
\item The operator $T_g:RM(p,0)\rightarrow RM(p,0)$ is bounded if and only if $g\in \mathcal{B}$.
\item The operator $T_g:RM(p,q)\rightarrow RM(p,q)$ is compact if and only if $g\in \mathcal{B}_{0}$.
\item The operator $T_g:RM(p,0)\rightarrow RM(p,0)$ is compact if and only if $g\in \mathcal{B}_{0}$.
\end{enumerate}
\end{theorem}
\begin{proof}
(1)	Assume that $g\in \mathcal{B}$. If $q<+\infty$, by Proposition~\ref{bergtypine}, there is a constant $C_{p}$ such that
	\begin{equation*}
	\begin{split}
	\rho_{p,q}^{q}(T_g(f))&\leq C_{p}^{q}\int_{0}^{2\pi}\left(\int_{0}^{1} |f(re^{i\theta})g'(re^{i\theta})|^{p} (1-r^2)^p\ dr\right)^{q/p}\frac{d\theta}{2\pi}\\
	&\leq C_{p}^{q} \|g\|_{\mathcal{B}}^{q}\int_{0}^{2\pi}\left(\int_{0}^{1} |f(re^{i\theta})|^{p}\ dr\right)^{q/p}\frac{d\theta}{2\pi}=C_{p}^{q} \|g\|_{\mathcal{B}}^{q}\, \rho_{p,q}(f)^{q},
	\end{split}
	\end{equation*}
	getting the boundedness of $T_{g}$ if $q<+\infty$. A similar argument works if $q=+\infty$.
	
	Conversely, assume that $T_{g}$ is bounded on $RM(p,q)$.  Fix $z\in \D$. By Proposition \ref{main-properties}(3), there is $C>0$ such that if $f\in B_{RM(p,q)}$, then
\begin{equation*}
\begin{split}
|g'(z)| |f(z)|&=|T_{g}(f)'(z)|
\leq \rho_{p,q}(T_{g}(f))\|\delta'_z\|_{(RM(p,q))^{\ast}}\\
&\leq \Vert T_{g}\Vert \, \rho_{p,q}(f)\|\delta'_z\|_{(RM(p,q))^{\ast}} \leq C \Vert T_{g}\Vert \, \rho_{p,q}(f) \|\delta_z\|_{(RM(p,q))^{\ast}}\frac{1}{1-|z|}.
	\end{split}
	\end{equation*}
We can choose $f\in RM(p,q)$, with $\rho_{p,q}(f)\leq 1$,  such that   $\|\delta_z\|_{(RM(p,q))^{\ast}}\leq 2 |f(z)|$. Therefore,
	$$
	|g'(z)|(1-|z|^{2})\leq C \Vert T_{g}\Vert \, \rho_{p,q}(f) \frac{\|\delta_z\|_{(RM(p,q))^{\ast}}}{|f(z)|}\leq 2C \Vert T_{g}\Vert ,
	$$	
	so that $g\in \mathcal{B}$ and (1) holds. 

\noindent (2) Assume that $g\in \mathcal{B}$. By (1) we know that $T_{g}$ is bounded from $RM(p,0)$ into $RM(p,\infty)$. If $f$ is a polynomial and $z\in \D$, then 
$$
(1-|z|^{2}) |T_{g}(f)'(z)|\leq \Vert g\Vert_{\mathcal B} \Vert f\Vert_{\infty}.
$$
That is $T_{g}(f)\in \mathcal B$ and, by Lemma \ref{bloch}$, T_{g}(f)\in RM(p,0)$. The density of the polynomials in $RM(p,0)$ and the boundedness of $T_{g}$ from $RM(p,0)$ into $RM(p,\infty)$ (by (1)) implies that $T_{g}(RM(p,0))\subset RM(p,0)$. 

Conversely, if $T_{g}$ is bounded on $RM(p,0)$ we can argue as in the proof of statement (1) using Proposition \ref{main-properties}(4) instead of Proposition \ref{main-properties}(3).

\noindent (3) and (4) We start by proving that if $T_g$ is compact in $RM(p,q)$, with $q<+\infty$, then $g\in\mathcal{B}_{0}$. Take $f\in RM(p,q)$, then 
\begin{align*}
\langle f,T_{g}^{\ast}(\delta'_{z})\rangle=\langle T_g(f),\delta'_{z}\rangle =g'(z)f(z)=g'(z)\langle f,\delta_z\rangle ,
\end{align*}
multiplying by $\frac{1-|z|}{\|\delta_z\|}$ we obtain
\begin{align*}
\langle f, T_{g}^{\ast}\left(\frac{\delta'_{z} (1-|z|)}{\|\delta_{z}\|}\right)\rangle =g'(z)(1-|z|) \langle f,\frac{\delta_{z}}{\|\delta_{z}\|}\rangle.
\end{align*}
Hence, it follows that
\begin{align*}
|g'(z)|(1-|z|) \left|\langle f,\frac{\delta_z}{\|\delta_{z}\|}\rangle \right|\leq \rho_{p,q}(f) \left\|T_{g}^{\ast}\left(\frac{\delta'_{z} (1-|z|)}{\|\delta_{z}\|}\right)\right\|_ {(RM(p,q))^{\ast}}.
\end{align*}
Taking supremum in $f\in B_{RM(p,q)}$, we have
\begin{align*}
|g'(z)|(1-|z|) \leq \left\|T_{g}^{\ast}\left(\frac{\delta'_{z} (1-|z|)}{\|\delta_{z}\|}\right)\right\|_ {(RM(p,q))^{\ast}}.
\end{align*}

We claim that   $\frac{\delta'_{z} (1-|z|)}{\|\delta_{z}\|}\xrightarrow{w^{\ast}} 0$ when $|z|\rightarrow 1$. Assuming the claim holds, the compactness of $T_g$ implies that $$\left\|T_{g}^{\ast}\left(\frac{\delta'_{z} (1-|z|)}{\|\delta_{z}\|}\right)\right\|_ {(RM(p,q))^{\ast}}\rightarrow 0,$$
as $|z|\to 1$, so that $g\in \mathcal{B}_{0}$.

Let us see the claim. 
Take $p$ a  polynomial. Then 
\begin{align*}
\frac{	|\delta'_{z}(p)| (1-|z|)}{\|\delta_{z}\|}\asymp \frac{|p'(z)|(1-|z|)}{(1-|z|)^{-\frac{1}{p}-\frac{1}{q}}}\lesssim \|p'\|_{\infty} (1-|z|)^{1+\frac{1}{p}+\frac{1}{q}}\rightarrow 0
\end{align*}
as $|z|\rightarrow 1$. The density of the polynomials in $RM(p,q)$ and the fact that 
$
\frac{\|\delta'_{z}\|(1-|z|)}{\|\delta_{z}\|}\lesssim 1,
$
show that
$
\frac{\delta'_{z} (1-|z|)}{\|\delta_{z}\|} \xrightarrow{w^{\ast}} 0
$
as $|z|\rightarrow 1$. So that the claim holds.

The same argument shows that if $T_{g}:RM(p,0)\to RM(p,0)$ is compact,  then $g\in \mathcal{B}_{0}$.

Assume now that $q=+\infty$. The compactness of $T_{g}:RM(p,\infty)\to RM(p,\infty)$ implies $g\in \mathcal B$ so that $T_{g}:RM(p,0)\to RM(p,0)$ is bounded and, clearly compact. Thus $g\in \mathcal{B}_{0}$.

Let us see that if $g\in \mathcal{B}_{0}$, then the operator $T_{g}$ is compact. Assume for the moment that $g$ is a polynomial. Take $\{f_k\}$ a  sequence in the unit ball of $ RM(p,q)$ uniformly convergent to $0$ on compact sets. Let $\varepsilon>0$. There is $N\in \N$  such that $|f_{k}(z)|\leq \varepsilon$ for all $|z|\leq \rho:=1-\varepsilon$ and $k\geq N$. Fix  $z=re^{i\theta}$. If $r\leq\rho$, then 
$$
\left|T_{g}f_{k}(r e^{i\theta} )\right|\leq \|g'\|_{\infty} \int_{0}^{r} |f_{k}(se^{i\theta})|\ ds\leq  \|g'\|_{\infty} \varepsilon 
$$
while, if $r> \rho$, then 
\begin{align*}
\left|T_{g}f_{k}(r e^{i\theta} )\right|\leq  \|g'\|_{\infty}  \int_{0}^{r} |f_{k}(se^{i\theta})|\ ds\leq  \|g'\|_{\infty}  \left(\int_{0}^{1} |f_{k}(se^{i\theta})|^{p}\ ds\right)^{1/p}.
\end{align*}
Therefore
\begin{align*}
\left(\int_{0}^{1} |T_{g}(f_{k})(re^{i\theta})|^{p}\ dr\right)^{1/p}\leq  \|g'\|_{\infty}  \left(\varepsilon \rho^{1/p} + (1-\rho )^{1/p}\left(\int_{0}^{1} |f_{k}(se^{i\theta})|^{p}\ ds\right)^{1/p} \right).
\end{align*}
Hence, for all $k\geq N$,
\begin{align*}
\rho_{p,q}(T_{g}(f_{k}))&\leq  \|g'\|_{\infty}  \varepsilon \rho^{1/p} +  \|g'\|_{\infty}  (1-\rho )^{1/p} \rho_{p,q}(f_{k})\\
&\leq  \|g'\|_{\infty}  (\varepsilon +(1-\rho)^{1/p})=\|g'\|_{\infty}  (\varepsilon +\varepsilon^{1/p}).
\end{align*}
Therefore $\lim_{k}\rho_{p,q}(T_{g}(f_{k}))=0.$ By Proposition~\ref{prop2comp}, $T_{g}$ is compact on $RM(p,q)$. 

If $g\in\mathcal{B}_{0}$, there is a sequence of polynomials $\{g_{n}\}$ such that $\lim_{n}\Vert g-g_{n}\Vert _{\mathcal{B}}=0$. Moreover, there is a constant $C=C(p)$ (see the proof of statement (1)) such that 
\begin{align*}
\|T_{g}-T_{g_{n}}\|= \|T_{g-g_{n}}\|\leq C_p \|g-g_{n}\|_{\mathcal{B}}\rightarrow 0.
\end{align*}
Since $T_{g_{n}}$ is compact for all $n$, then so is $T_{g}$.

The same argument works in $RM(p,0)$ so that we are done.
\end{proof}

\begin{remark}
	Let $1\leq p<+\infty$. If $g\in\mathcal{B}_{0}$ then $T_{g}(RM(p,\infty))\subset RM(p,0)$. 
Indeed, if  $f\in RM(p,\infty)$ and $h$ is  a polynomial, then for all $r\in(0,1)$ we have that 
		\begin{align*}
		|T_{h}f(re^{{i\theta}})|\leq \Vert h'\Vert_{\infty}\ \rho_{p,\infty}(f).
		\end{align*}
		That is, $T_{h}f\in H^{\infty}\subset RM(p,0)$. 
		Hence, using density of polynomials in $\mathcal{B}_{0}$ and the estimate $\|T_{g}\|\leq C_{p} \|g\|_{\mathcal{B}}$ (see the proof of statement (1) in above theorem), we can prove that if $g\in\mathcal{B}_{0}$ then $T_{g}(RM(p,\infty))\subset RM(p,0)$, because $RM(p,0)$ is closed in $RM(p,\infty)$ and
		\begin{align*}
		\|T_{g}-T_{h_{n}}\|=\|T_{g-h_{n}}\|\leq C_{p} \|g-h_{n}\|_{\mathcal{B}}\rightarrow 0,
		\end{align*}
		where $h_{n}$ are polynomials such that $\|g-h_{n}\|_{\mathcal{B}}\rightarrow 0$.
\end{remark}

However, the reverse implication does not hold as next example shows.  That is, the compactness cannot be characterized  by the property of sending the big-O space into the little-o space, 
despite what happens in other spaces of holomorphic functions (see, i.e., \cite{Arevalo-Contreras-Piazza} for mixed norm spaces, \cite{BCHMP} for weighted Banach spaces, and \cite{BCDMPS} for the Bloch space and BMOA).

\begin{example}\label{ConjInt}
	Let $1\leq p<+\infty$ and $g(z)=-\log(1-z)$. Then  $g\in \mathcal{B}\setminus\mathcal{B}_{0}$ and $T_{g}(RM(p,\infty))\subset  RM(p,0)$.
	\end{example}
\begin{proof}
Fix $f\in RM(p,\infty)$ such that $\rho_{p,\infty}(f)\leq 1$. By Proposition~\ref{rhoLPEQ}, in order to prove that $T_{g}(f)\in RM(p,0)$, it is enough to show
	\begin{align*}
	\lim\limits_{\rho\rightarrow 1^{-}} \sup_{\theta} \left(\int_{\rho}^{1}|f(re^{i\theta})g'(re^{i\theta})|^{p}(1-r)^{p}\ dr\right)^{1/p}=0.
	\end{align*}
	
Suppose by contradiction that there are a constant $c>0$ and sequences $\{\rho_{k}\}\rightarrow 1$ and $\{\theta_{k}\}$  in $(-\pi,\pi)$ such that
	\begin{align*}
	\left(\int_{\rho_{k}}^{1} |f(re^{i\theta_k})|^{p}  \frac{(1-r)^{p}}{|1-re^{i\theta_k}|^{p}}\ dr\right)^{1/p}>c.
	\end{align*}
	
Notice that the sequence $\{\theta_k\}$ must converge to $0$. Indeed, using that $|1-e^{i\theta_k}|\leq 2|1-re^{i\theta_k}| $ and that $\rho_k\rightarrow 1$ we have
	\begin{align*}
	c<\left(\int_{\rho_{k}}^{1} |f(re^{i\theta_k})|^{p}  \frac{(1-r)^{p}}{|1-re^{i\theta_k}|^{p}}\ dr\right)^{1/p}<2\left(\int_{0}^{1} |f(re^{i\theta_k})|^{p} \frac{(1-\rho_{k})^{p}}{|1-e^{i\theta_{k}}|^{p}}\ dr\right)^{1/p}<2\frac{(1-\rho_{k})}{|1-e^{i\theta_{k}}|}
	\end{align*}
	so that  it holds that $\theta_k\rightarrow 0$.
	
	\noindent  Claim 1.
		There is  $\delta>0$ such that if  $\theta\in \left[\frac{-\pi}{4},\frac{\pi}{4}\right]\setminus\{0\}$ and $1>r>1-\delta |\theta|$, then $\left|\frac{1-r}{1-re^{i\theta}}\right|<c/4$.

	\noindent  Proof of Claim 1.
		Notice $\frac{1}{2}<\frac{1-\cos(\theta)}{\theta^2/2}<1$ for $\theta\in [-\pi/4,\pi/4]\setminus\{0\}$. Therefore if $r>1-\delta|\theta|$, then
		\begin{align*}
		\frac{(1-r)^2}{(1-r)^2+2r(1-\cos(\theta))}<\frac{1}{1+\frac{1-\delta|\theta|}{2\delta^2}}<\frac{1}{1+\frac{1-\delta\frac{\pi}{4}}{2\delta^2}}<\frac{c^2}{16}
		\end{align*}
		if $\delta$ is small enough and Claim 1 holds. 

By Claim 1 and the fact that $\rho_{p,\infty}(f)\leq 1$ we have 
	\begin{align*}
	\left(\int_{1-\delta|\theta_k|}^{1} |f(re^{i\theta_k})|^p \frac{(1-r)^p}{|1-re^{i\theta_k}|^p}\ dr\right)^{1/p}<\frac{c}{4}.
	\end{align*}
Therefore,
	\begin{align*}
	\int_{\rho_{k}}^{1} |f(re^{i\theta_k})|^{p}  \frac{(1-r)^{p}}{|1-re^{i\theta_k}|^{p}}\ dr-\int_{1-\delta|\theta_{k}|}^{1} |f(re^{i\theta_k})|^{p}  \frac{(1-r)^{p}}{|1-re^{i\theta_k}|^{p}}\ dr>c^p-\frac{c^p}{4^p}\geq\frac{3}{4}c^p>0,
	\end{align*}
so that $1-\delta|\theta_k|>\rho_k$.\newline
	Now, using again Claim 1, it is obtained that
	\begin{align*}
	&\int_{\rho_k}^{1-\delta |\theta_k|} |f(re^{i\theta_k})|^{p} \ dr\geq \int_{\rho_k}^{1-\delta |\theta_k|} |f(re^{i\theta_k})|^{p} \frac{(1-r)^p}{|1-re^{i\theta_{k}}|^{p}}\ dr\\
	&\qquad=\int_{\rho_k}^{1} |f(re^{i\theta_k})|^{p} \frac{(1-r)^p}{|1-re^{i\theta_{k}}|^{p}}\ dr-\int_{1-\delta |\theta_k|}^{1} |f(re^{i\theta_k})|^{p} \frac{(1-r)^p}{|1-re^{i\theta_{k}}|^{p}}\ dr>\frac{3}{4}c^p.
	\end{align*}
	
 	\noindent  Claim 2.
		There is $M>\delta$ and $k_{0}$ such that, if $\rho_k< 1-M|\theta_k|$, then
		\begin{align*}
		\left(\int_{\rho_k}^{1-M|\theta_k|} |f(re^{i\theta_k})|^{p}  dr\right)^{1/p}<\frac{c}{3}
		\end{align*}
		for  $k>k_0$.

	\noindent  Proof of Claim 2.
		By Proposition~\ref{main-properties}(2), there is a constant $C_{1}$ such that $|f'(w)|\leq C_{1}(1-|w|)^{-1-\frac{1}{p}}$, for all $w\in \D$. Take  $M>\max\left\{\delta,\frac{4C_1}{c}\right\}$ and $k$ such that $M|\theta_k|<1$ . Then 
		\begin{align*}
		\left(\int_{0}^{1-M|\theta_k|} |f(r)-f(re^{i\theta_k})|^{p}\ dr\right)^{1/p} &< \left(\int_{0}^{1-M|\theta_k|} \sup_{w\in [r,re^{i\theta_k}]}|f'(w)|^{p} |1-e^{i\theta_k}|^{p}\ dr\right)^{1/p}\\
		&\leq C_1|1-e^{i\theta_k }|\left(\int_{0}^{1-M|\theta_k|}\frac{dr}{(1-r)^{p+1}}\right)^{1/p}\\
		&<C_1|1-e^{i\theta_k }| \frac{1}{p^{1/p}M|\theta_k|}\leq\frac{C_1}{M}<\frac{c}{4}.
		\end{align*}
By the integrability of $|f(r)|^p$  in the interval $[0,1)$ and the fact that $\rho_k\rightarrow 1^-$, there exist $k_0$ such that for all $k>k_0$ we have 
		\begin{align*}
		&\left(\int_{\rho_k}^{1-M|\theta_k|} |f(re^{i\theta_k})|^{p}\ dr\right)^{1/p}\\
		&\qquad \leq \left(\int_{\rho_k}^{1-M|\theta_k|} |f(re^{i\theta_k})-f(r)|^{p}\ dr\right)^{1/p}+\left(\int_{\rho_k}^{1-M|\theta_k|} |f(r)|^{p}\ dr\right)^{1/p}
		\leq \frac{c}{4}+\frac{c}{13}<\frac{c}{3}
		\end{align*}
and Claim 2 holds.
	
	If $\rho_k<1-M|\theta_k|$, it follows
	\begin{align*}
	\int_{1-M|\theta_k|}^{1-\delta|\theta_k|} |f(re^{i\theta_k})|^{p}\ dr	&= \int_{\rho_{k}}^{1-\delta|\theta_k|} |f(re^{i\theta_k})|^{p}\ dr - \int_{\rho_k}^{1-M|\theta_k|} |f(re^{i\theta_k})|^{p}\ dr \\
	&>\frac{3}{4}c^p-\frac{c^p}{3^p}=c^p\left(\frac{3}{4}-\frac{1}{3^p}\right)\geq\frac{5c^p}{12}.
	\end{align*}
	If $\rho_k>1-M|\theta_k|$, we obtain
	\begin{align*}
	\int_{1-M|\theta_k|}^{1-\delta|\theta_k|} |f(re^{i\theta_k})|^{p}\ dr > \int_{\rho_k}^{1-\delta|\theta_k|} |f(re^{i\theta_k})|^{p}\ dr>\left(1-\frac{1}{4^p}\right)c^p
	>\frac{3c^p}{4}>\frac{5c^p}{12}.
	\end{align*}
Therefore, there exists $r_k\in (1-M|\theta_{k}|,1-\delta|\theta_k|)$ such that
	\begin{align*}
	|f(r_ke^{i\theta_k})|^{p} (M-\delta)|\theta_k|\geq \int_{1-M|\theta_k|}^{1-\delta|\theta_k|} |f(re^{i\theta_k})|^{p}\ dr>\frac{5c^p}{12}.
	\end{align*}
Thus
	\begin{align*}
	|f(r_ke^{i\theta_k})| (1-r_k)^{1/p}>\frac{5^{1/p}\delta^{1/p} c}{12^{1/p}(M-\delta)^{1/p}}
	\end{align*}
	what contradicts Proposition \ref{main-properties}(6). 
	\end{proof}

\section{Weak compactness of the integration operator $T_g$} 
\addtocontents{toc}{\protect\setcounter{tocdepth}{1}}
It is well-known that any weakly compact integration operator on the Hardy space $H^{1}(=RM(\infty,1))$ is compact (see \cite{La-Mi-Ni}). Since  the Bergman space $A^{1}(=RM(1,1))$ is isomorphic to $\ell _{1}$ (see \cite[p. 89]{wojtaszczyk_banach_1991}) and then it has the Schur property,  it also holds that if $T_{g}$ is weakly compact on $A^{1}$ then it is compact. In this section, we will show that this happens in other spaces of average radial integrability but not in all of them. When the weak compactness does not coincide with the compactness we will provide different characterizations.

Since $RM(p,q)$ is reflexive if either $1<p,q<+\infty$ or $p=+\infty$ and $1<q<+\infty$, the problem we are dealing with in this section it is only interesting in the next three cases:
\begin{itemize}
\item $1\leq p\leq +\infty$ and $q=+\infty$;
\item $p=1$ and $1\leq q\leq +\infty$;
\item $1\leq p\leq +\infty$ and $q=1$.
\end{itemize} 

 There is a useful characterization of the weak compactness of $T_{g}$ in terms of the norm convergence of certain convex combinations.

\begin{lemma}\label{lemmaweakcompact1q}
Let $1\leq p, q, \tilde p, \tilde q \leq +\infty$ and $X$ a Banach space.
\begin{enumerate}
\item Let $T:RM(p,q)\to X$ be a linear and bounded operator. Assume that  for every sequence $\{f_n\}$ in the unit ball of $RM(p,q)$  convergent to $0$ uniformly on compact sets of $\D$ satisfies that there exist $g_k\in \mathrm{co}\{f_k,f_{k+1},\dots\}$ such that $\Vert T g_k\Vert \rightarrow 0$ when $k\rightarrow \infty$. Then $T$ is weakly compact.

\item Assume that $T_g:RM(p,q)\rightarrow RM(\tilde p,\tilde q)$ is bounded.
Then $T_g:RM(p,q)\rightarrow RM(\tilde p,\tilde q)$ is weakly compact if and only if every sequence $\{f_n\}$ in the unit ball of $RM(p,q)$  convergent to $0$ uniformly on compact sets of $\D$ satisfies that there exist $g_k\in \mathrm{co}\{f_k,f_{k+1},\dots\}$ such that $\rho_{\tilde p,\tilde q}(T_{g} g_k)\rightarrow 0$ when $k\rightarrow \infty$.
\end{enumerate}
	\end{lemma}
\begin{proof} 
Let us begin with (1).  Assume by contradiction that $T$ is not weakly compact. Then there is a bounded sequence $\{f_n\}$ such that $\{T f_{n}\}$ does not have weakly convergent subsequences.
	Applying Montel's theorem, there is a holomorphic function $f$ and a subsequence $\{f_{n_k}\}$ such that it converges uniformly to $f$ on compact sets of $\D$. By Fatou's Lemma, it holds that $f\in RM(p,q)$. Consider the bounded sequence $\{h_{k}\}:=\{f_{n_k}-f\}$. Clearly it converges uniformly to $0$ on compact sets of $\D$. Since $\{T h_{k}\}$ does not converge weakly to zero, there are $\lambda\in X^{\ast}$, $\delta>0$, and a subsequence $\{T h_{k_j}\}$ such that 
$$\Re\left(\lambda\left(Th_{k_j}\right)\right)\geq \delta >0.$$
	By our assumption, there exists $g_j\in \textrm{co}\{h_{k_j},h_{k_{j+1}},\dots\}$ such that $\Vert T g_{j}\Vert _{X}\rightarrow 0$. But, since 
	$$\Re \lambda (T g_j)=\sum_{l=j}^{\infty} \alpha_{l,j} \Re \lambda \left(T h_{k_{l}}\right)\geq \delta >0,$$ for certain coefficients $0\leq \alpha_{k,j}\leq 1$, with $\sum_{j=k}^{\infty} \alpha_{k,j}=1$, and, for each $k$, the set $\{j\geq k:\, \alpha_{k,j}\neq 0\}$ is finite, we obtain a contradiction because
	$$ 0<\delta\leq \Re \lambda( Tg_j )\leq |\lambda (T g_j)|\leq \|\lambda\| \, \Vert T g_j\Vert_{X}.$$

Let us prove (2). By (1), we just have to check one implication. Assume that $T_g:RM(p,q)\rightarrow RM(\tilde p,\tilde q)$ is weakly compact. 
	Let $\{f_n\}\subset B_{RM(p,q)}$ be a sequence that converges uniformly to $0$ on compact sets of $\D$. By the very definition of integration operator,  we also have that $T_{g} f_n$ converges  to $0$ uniformly on compact sets of the unit disc. By the weak compactness of $T_{g}$, there exists a subsequence $\{T_{g}f_{n_k}\}$ that converges weakly to some $h\in \mathcal{H}(\D)$. Since the convergence in the weak topology implies pointwise convergence, we have that $h=0$. 
	Therefore $\{T_{g}f_{n_k}\}$ converges  weakly to $0$. By  \cite[Corollary on p. 28]{wojtaszczyk_banach_1991}, we obtain that there exists $g_{k}\in\textrm{co}\{f_{n_{k}},f_{n_{k+1}},\dots\} \subset\textrm{co}\{f_k,f_{k+1},\dots\}$ such that $\rho_{\tilde p,\tilde q}(T_g g_k)\rightarrow 0$.
\end{proof}

\subsection{The case $q=+\infty$}Unlike what happens in other spaces of holomorphic functions (see, i.e., \cite{Arevalo-Contreras-Piazza} for mixed norm spaces, \cite{BCHMP} for weighted Banach spaces, and \cite{BCDMPS} for the Bloch space and BMOA),  Example \ref{ConjInt} shows that the compactness cannot be characterized  by the property of sending the big-O space into the little-o space. Nevertheless, this property characterizes the weak compactness in the spaces $RM(p,\infty)$ for $1<p<+\infty$. 

\begin{theorem}\label{weakcompactnessRM-p-infty}
	Let $1<p<+\infty$ and $g\in \mathcal B$. The following are equivalent:
	\begin{enumerate}
		\item[(1)] $T_{g}(RM(p,\infty))\subset RM(p,0)$. 
		\item[(2)] $T_g:RM(p,0)\rightarrow RM(p,0)$ is weakly compact.
		\item[(3)] $T_g:RM(p,\infty)\rightarrow RM(p,\infty)$ is weakly compact.
	\end{enumerate}
\end{theorem}
\begin{proof} Let us recall that given a Banach space $X$ and a bounded operator $T:X\to X$ it holds that $T$ is weakly compact if and only if $T^{**}:X^{**}\to X$ (see \cite[Theorem 6, p. 52]{wojtaszczyk_banach_1991}) if and only if $T^{**}:X^{**}\to X^{**}$ is weakly compact. 

Let us consider the bounded operator $T_{g}:RM(p,0)\rightarrow RM(p,0)$. A standard argument using Theorem \ref{DualRMpinfty} gives that the next diagram is commutative:
$$
\xymatrix{  
		(RM(p,0))^{\ast\ast}\ar[r]^{(T_{g})^{\ast\ast}}\ar[d]_{I^{\ast\ast}} & (RM(p,0))^{\ast\ast}\ar[d]_{I^{\ast\ast}}  \\	
		RM(p,\infty)  \ar[r]^{T_{g}} & RM(p,\infty) \\
}
$$
Since $I^{\ast\ast}$ is an isomorphism, above general results give the theorem.
\end{proof}

A similar result to above theorem for $p=+\infty$ was obtained in \cite{CPPR}. Namely,  they proved that $T_g$ is weakly compact on $H^{\infty}$ if and only if it is weakly compact on the disc algebra and if and only if $T_{g}$ sends $H^{\infty}$ into the disc algebra. Let us recall that the disc algebra is the closure of the polynomials in $H^{\infty}$ in an analogous way to the couple $RM(p,0)$ and $RM(p,\infty)$. 

Next example shows that Theorem \ref{weakcompactnessRM-p-infty}  does not hold for $p=1$.

\begin{example}
	Let $g(z)=-\log(1-z)\in \mathcal{B}\setminus \mathcal{B}_{0}$. Then $T_g:RM(1,\infty)\rightarrow RM(1,0)$ fixes a copy of $\ell^1$. In particular, $T_g:RM(1,\infty)\rightarrow RM(1,0)$ is not weakly compact.
\end{example}
\begin{proof} By Example \ref{ConjInt}, $T_g:RM(1,\infty)\rightarrow RM(1,0)$ is a bounded operator. Take $\beta \geq 2$ a natural number and write $\delta:=\frac{2}{3}\frac{\beta}{1+\beta}$. Notice that the sequence $\{T_{g}(\beta^{n}z^{\beta^n})\}$ converges to zero uniformly on compact subsets of $\D$.  Consider the sequence of functions $f_n:[0,1)\to \C$ given by $f_{n}(r):=T_{g}(\beta^{n}z^{\beta^n})(r)$ for $r\in [0,1)$. Notice that 
\begin{align*}
	\int_{0}^{1} \left|T_{g}(\beta^{n}z^{\beta^n})(r)\right|\, dr= \left|\int_{0}^{1}\int_{0}^{r} \beta^{n}\frac{u^{\beta^n}}{1-u}\, du\, dr\right|=\left|\int_{0}^{1}\int_{u}^{1} \beta^{n}\frac{u^{\beta^n}}{1-u}\, dr\, du\right|
	=\frac{\beta^n}{1+\beta^n}\geq \frac{3}{2}\delta
	\end{align*}
	for all $n\in\N$. 
That is,  $1\geq\| f_n\|_{1}\geq 3\delta/2$ for $n\in\N$.
Since all the functions $f_n$ are integrable in $[0,1)$ and goes to zero uniformly on compacta of such interval we can choose $r_1\in [0,1)$ and $n_2$ such that 
	\begin{align*}
	\int_{r_{1}}^{1} |f_1|\ dr<\frac{\delta}{2} \quad \textrm{ and } \quad 	\int_{0}^{r_1} |f_{n_2}|\ dr <\frac{\delta}{4}.
	\end{align*}
Repeating the argument we choose $r_{2}\in (r_1,1)$ and $n_3$ such that such that 
	\begin{align*}
	\int_{r_{2}}^{1} |f_{n_2}|\ dr<\frac{\delta}{4}\quad \textrm{ and } \quad 	\int_{0}^{r_2} |f_{n_3}|\ dr <\frac{\delta}{4}.
	\end{align*}
Continuing inductively we obtain a subsequence $\{f_{n_k}\}$ and a sequence of disjoint intervals $\{I_k\}=\{(r_{k-1},r_{k})\}$, setting $r_{0}=0$, such that 
	\begin{align*}
	\int_{I_{k}} |f_{n_{k}}|\ dr >\delta \quad \textrm{ and } \quad 
	\int_{\cup_{j\neq k}I_{j}} |f_{n_{k}}|\ dr<\frac{\delta}{2}.
	\end{align*}
Now, we consider the operator $\Theta: \ell^1\rightarrow RM(1,\infty)$ given by $\Theta(\{\alpha_{k}\})=\sum_{k=1}^{\infty} \alpha_{k}\beta^{n_k} z^{\beta^{n_k}}$. By  Proposition \ref{lacunary}, 
\begin{align*}
	\rho_{1,\infty}(\Theta(\{\alpha_{k}\}))=\rho_{1,\infty}\left(\sum_{k=1}^{\infty} \alpha_{k}\beta^{n_k} z^{\beta^{n_k}}\right)\asymp \left(\sum_{k=0}^{\infty}\frac{|\alpha_k| \beta^{n_k}}{\beta^{n_k}+1}\right)\leq \|\{\alpha_{k}\}\|_{\ell^1}.
	\end{align*}
Therefore, the boundedness of $T_g$ implies that $T_g\circ\Theta:\ell^{1}\rightarrow RM(1,0)$ is continuous.	On the other hand, 
\begin{align*}
&\int_{0}^{1} \left|\sum_{k=1}^{\infty}\alpha_{k}f_{n_k}(r)\right|\ dr\geq \sum_{k=1}^{\infty}\int_{I_k} \left(
|\alpha_{k}||f_{n_k}(r)|-\sum_{j\neq k} |\alpha_{j}||f_{n_j}(r)|\right)\ dr\\
&\geq \delta \sum_{k=1}^{\infty}|\alpha_{k}|-\sum_{j=1}^{\infty}|\alpha_j|\sum_{k\neq j} \int_{I_k} |f_{n_k}(r)|\ dr\geq \delta \sum_{k=1}^{\infty}|\alpha_{k}|-\frac{\delta}{2}\sum_{k=1}^{\infty}|\alpha_{k}|=\frac{\delta}{2}\|\{\alpha_{k}\}\|_{\ell_{1}}.
\end{align*}
	Hence, we have that $T_g:RM(1,\infty)\rightarrow RM(1,0)$ fixes a copy of $\ell^1$ and we are done.
	\end{proof}

\subsection{The case $p=1$} 
We start with  a characterization of the weak compactness of the operator $T_g: RM(1,q)\rightarrow RM(1,q)$ in terms of non-fixing copies of $\ell^{1}$. We need the following lemma which probably is well-known by specialist but  we could not find any reference so that we include the proof for the sake of completeness.

\begin{lemma}\label{lemma-copy-l1}
Let $X$ be a Banach space and $\mu$ a positive and finite measure on $\Omega$. If $T:X\to L^{1}(\mu)$ is bounded and not weakly compact, then it fixes a copy of $\ell^{1}$. 
\end{lemma}
\begin{proof}
Since $T(B_{X})$ is not relatively weakly compact, by \cite[p. 93, Corollary]{Diestel}, there exists a sequence $\{f_{n}\}$  in $T(B_{X})$ which is equivalent to the basis of $\ell^{1}$. That is, there is a positive constant $\delta$ such that
$$
\Vert \sum_{n}\alpha_{n} f_{n}\Vert \geq \delta \sum_{n}|\alpha_{n} |
$$
for all sequences $\{\alpha_{n}\}$ of complex  numbers.
Take $x_{n}\in B_{X}$ such that $T(x_{n})=f_{n}$. Then
$$
 \sum_{n}|\alpha_{n} |\geq \Vert \sum_{n}\alpha_{n} x_{n}\Vert \geq  \frac{1}{||T||}\Vert \sum_{n}\alpha_{n} f_{n}\Vert \geq \frac{\delta}{||T||} \sum_{n}|\alpha_{n} |
$$
for all sequences $\{\alpha_{n}\}$ of complex  numbers and we are done.
\end{proof}

\begin{proposition}\label{Prop: weak-compact-1-q}
	Let $1<q<+\infty$ and $g\in\mathcal{B}$. The following assertions are equivalent:
	\begin{enumerate}
		\item $T_g: RM(1,q)\rightarrow RM(1,q)$ is weakly compact.
		\item $T_g: RM(1,q)\rightarrow RM(1,1)$ is compact.
		\item  $T_g: RM(1,q)\rightarrow RM(1,1)$ is weakly compact.
		\item $T_g: RM(1,q)\rightarrow RM(1,1)$ does not fix a copy of $\ell^1$.
		\item $T_g: RM(1,q)\rightarrow RM(1,q)$ does not fix a copy of $\ell^1$.
	\end{enumerate}
\end{proposition}
\begin{proof} It is obvious that (1) implies (5).   
Bearing in mind  the following commutative diagram 
	$$
	\xymatrix{
		RM(1,q) \ar[r]^{T_g} \ar@/_2pc/[rr]_{T_g} & RM(1,q) \ar@{^{(}->}[r] & RM(1,1)
	}$$
	it is clear that (5) implies (4). Notice that $RM(1,1)=A^1$ and it is isomorphic to $\ell^{1}$ \cite[Theorem 11, p. 89]{wojtaszczyk_banach_1991}. By Lemma \ref{lemma-copy-l1}, if $T_g: RM(1,q)\rightarrow RM(1,1)$ does not fix a copy of $\ell^1$, it is weakly compact. In addition, since $RM(1,1)$ is isomorphic to $\ell^1$, it has the Schur property and it must be compact. Thus, 
(4) implies (3) and (3) implies (2). Therefore, it remains to show that (2) implies (1). 
	
Assume that $T_g: RM(1,q)\rightarrow RM(1,1)$ is compact. Let $\{f_{n}\}\subset B_{RM(1,q)}$ be a sequence that converges uniformly to $0$ on compact sets of $\D$. Then, the compactness implies that $\rho_{1,1}(T_{g}f_n)\rightarrow 0$. 
	
The value $H_n(\theta):=\int_{0}^{1} |T_g f_{n} (re^{i\theta})|\, dr$ is finite for almost every $\theta$. 
Since $g\in \mathcal{B}$, there is a constant $C>0$ such that $\|H_n\|_{L^q(\mathbb{T})} =\rho_{1,q}(T_{g}(f))\leq C$. Moreover, $\lim_{n}\|H_n\|_{L^1(\mathbb{T})}=0$. Therefore, we obtain a subsequence $\{H_{n_k}\}$ such that $H_{n_k}\rightarrow 0$ weakly in $L^q(\mathbb{T})$.
	Hence, there is $F_{k}\in \textrm{co}\{H_{n_k},H_{n_{k+1}},\dots\}$ such that $\|F_k\|_{L^{q}(\mathbb{T})}\rightarrow 0$ (see \cite[Corollary on p. 28]{wojtaszczyk_banach_1991}).
Write
	\begin{align*}
	F_{k}=\sum_{j=k}^{\infty} \alpha_{k,j} H_{n_j},
	\end{align*}
	where $\alpha_{k,j}\geq 0$, $\sum_{j=k}^{\infty} \alpha_{k,j}=1$, and, for each $k$, the set $\{j\geq k:\, \alpha_{k,j}\neq 0\}$ is finite. The functions
	\begin{align*}
	g_{k}:=\sum_{j=k}^{\infty} \alpha_{k,j} f_{n_j},
	\end{align*}
belong to $RM(1,q)$ and  
	\begin{align*}
	\int_{0}^{1} |T_g g_k(re^{i\theta})|\ dr\leq \sum_{j=k}^{\infty} \alpha_{k,j} H_{n_j}=F_{k}(\theta).
	\end{align*}
It follows that $\rho_{1,q}(T_g g_k)\rightarrow 0$, as $k\rightarrow\infty$.
	Using Lemma~\ref{lemmaweakcompact1q} we conclude that $T_g:RM(1,q)\rightarrow RM(1,q)$ is weakly compact. 
\end{proof}

The main result of this section provides a characterization of the weak compactness of the operator $T_{g}: RM(1,q)\rightarrow RM(1,q)$ in terms of the symbol $g$. 
For this purpose, we introduce a pointwise version of the little Bloch space.
\begin{definition}\label{def:weaklylittleBloch}
The weakly little Bloch space, denoted by $\mathcal{B}_{0,w}$, is the closed subspace of $\mathcal{B}$ consisting of analytic functions $f\in \mathcal{B}$ with 
	\begin{align*}
	\lim_{r\to 1}(1-r^2)|f'(re^{i\theta})|=0,
	\end{align*}
	for almost every $e^{i\theta}\in \T$. 
\end{definition}

As  far as we know, this space appeared firstly in \cite{Pavlovic}. We are going to prove that $T_g\colon RM(1,q)\to RM(1,q)$ is weakly compact if and only $g\in \mathcal{B}_{0,w}$.
Next theorem provides one of the implications. A preliminary lemma is needed.

\begin{lemma}\label{Bcdelta}
Given $B$, $c>0$ there exists $\delta\in (0,1/2)$ such that for $g\in\mathcal{B}$ satisfying 
$|g'(z)|(1-|z|)\le B,$ for all $z\in \D$,
$\eta\in (0,1/2)$ and $e^{ia}\in\T$ satisfying
\begin{equation}\label{grande2c}
|g'((1-\eta)e^{ia})| \eta>2c,
\end{equation}
we have 
$$
|g'(re^{i\theta})|>\frac{c}{\eta},
$$
whenever $|r-(1-\eta)|<\delta\eta$ and $|\theta-a|<\delta\eta$.
\end{lemma}
\begin{proof}
It is not difflcult to show
(see for instance the proof of \cite[Theorem 5.5]{duren_theory_2000}) 
that if $g\in \mathcal{B}$ satisfies \eqref{BlochB} for $B>0$, then it also satisfies
\begin{equation}\label{segundader}
|g''(z)|\le \frac{4B}{(1-|z|)^2}, \qquad\text{for all $z\in \D$.}
\end{equation}
Assume now that $\eta\in (0,1/2)$ and $e^{ia}\in\T$ satisfy \eqref{grande2c}, and pick any 
$\delta\in (0,1/2)$. If 
$|r-(1-\eta)|<\delta\eta$ and $|\theta-a|<\delta\eta$, then 
$|re^{i\theta}-(1-\eta)e^{ia}|\le 2\eta\delta$
and every point $w$ in the segment joining $re^{i\theta}$ and 
$(1-\eta)e^{ia}$ has module $|w|\le (1-\eta)+\delta\eta$.
Hence, for all these $w$, we have
$$
|g''(w)|\le \frac{4B}{(1-\delta)^2\eta^2}\,
$$ 
and, by the mean value inequality,
$$
|g'((1-\eta)e^{ia})-g'(re^{i\eta})|\le \frac{8B\eta\delta}{(1-\delta)^2\eta^2}\le \frac{32 B \delta}{\eta}
\le\frac{c}{\eta},
$$
if $\delta\le c/32 B$. The lemma follows since, by \eqref{grande2c},
$$
|g'(re^{i\theta})|\ge |g'((1-\eta)e^{ia})|-\frac{c}{\eta} > \frac{2c - c}{\eta}=\frac{c}{\eta}.
$$
\end{proof}

\begin{theorem}\label{technical result}
Let $1<q<+\infty$ and $g\in\mathcal{B}\smallsetminus \mathcal{B}_{0,w}$. Then the operator
$$
R_{g}:RM(1,q)\to L^{1}([0,1)\times \T)
$$
defined by 
$$
R_g(f)(r,e^{i\theta}):=f(re^{i\theta})g'(re^{i\theta} )(1-r),\qquad r\in [0,1),\quad e^{i\theta}\in\T,
$$
is not weakly compact.
\end{theorem}
\begin{proof}
We will denote by $m$ both the Lebesgue measure on $[0,1)$ and the arc length mea\-sure on $  \T$,
and by $m_2$ the product  measure $m_2=m\otimes m$ on $[0,1)\times \T$. 
In order to prove that $R_g$ is not weakly compact, and using the Dunford-Pettis Theorem (see, e.g. \cite[Theorem 15, p. 76]{Diestel-Uhl} or \cite[p. 137]{wojtaszczyk_banach_1991}), 
we need to show that the image by $R_g$ of the unit ball of $RM(1,q)$ is not uniformly integrable. This will be done if we show the existence of two constants $C$, $\a>0$ such that,
for every $\varepsilon>0$, there exits $f\in RM(1,q)$ and a measurable set $D\subset [0,1)\times \T$
such that
\begin{equation}\label{noequiint}
\text{(a) }\ m_2(D) <\varepsilon,\qquad
\text{(b) }\ \rho_{1,q}(f)\le C, \qquad \text{and} \qquad
\text{(c) }\ \int_D |R_g(f)|\, dm_2 > \a.
\end{equation}

The condition $g\in\mathcal{B}\smallsetminus \mathcal{B}_{0,w}$ yields the existence of two constants 
$B$, $c>0$ and a measurable set $A\subset \T$ of positive measure such that
\begin{equation}\label{BlochB}
|g'(z)|(1-|z|)\le B, \qquad\text{for all $z\in \D$,}
\end{equation}
and
\begin{equation}\label{NoBloch0w}
\limsup_{r\to 1^-} |g'(re^{i\theta})|(1-r) > 2c, \qquad \text{for every $e^{i\theta}\in A$.}
\end{equation}
We assume that $m(A)>\beta>0$. Finally take $M>1$ big enough (to be determined later).

In order to get the conditions in \eqref{noequiint}, fix $\varepsilon\in(0,1)$. 
For every $e^{ia}\in A$, there exists
$\varepsilon_a\in (0,\varepsilon/4\pi)$ such that 
\begin{equation}\label{eaEnA}
|g'((1-\varepsilon_a)e^{ia})|\varepsilon_a>2c.
\end{equation}
Recall that $M>1$ is big enough and consider, for every $e^{ia}\in A$,
the open arc
$$
J_a:=\{ e^{it} : t\in (a-M\varepsilon_a,a+M\varepsilon_a)\}.
$$
The family of all these arcs is a covering of $A$. So passing first
through a compact set $K\subset A$ with $m(K)>\beta$ in order to get a finite covering and then
using
Hardy-Littlewood covering lemma (see for instance \cite[Lemma 7.3]{rudin_real_1987}) there exist
$N\in\N$ and $a_1$, $a_2$, \dots, $a_N$ such that $\{ J_{a_k} : k=1,2,\ldots,N \}$
is a family of pairwise disjoint arcs with
\begin{equation}\label{covering}
\sum_{k=1}^N m(J_{a_k}) > \beta/3.
\end{equation}

We will put $\varepsilon_k$ and $J_k$ instead of $\varepsilon_{a_k}$ and $J_{a_k}$ respectively.
From \eqref{covering} we get 
\begin{equation}\label{sumaes}
\sum_{k=1}^N \varepsilon_k >\frac{\beta}{6M}.
\end{equation}
We will also consider the arcs 
$$
L_k:=\{ e^{it} : t\in [a_k-\delta\varepsilon_k,a_k+\delta\varepsilon_k]\},
$$
where $\delta$ is the one in Lemma \ref{Bcdelta}, and the subsets of $[0,1)\times \T$,
$$
D_k:=[1-\varepsilon_k-\delta\varepsilon_k,1-\varepsilon_k+\delta\varepsilon_k]
\times L_k,\qquad D=\bigcup_{k=1}^N D_k.
$$
Observe that $D$ is a compact subset of  $(1-\varepsilon/2\pi,1)\times \T$ and therefore
$$
m_2(D)< \frac{\varepsilon}{2\pi} 2\pi=\varepsilon.
$$
This yields \eqref{noequiint}(a).

Let us define the function $f$. Consider, for $z\in\D$,
\begin{equation}\label{definouf}
u_k(z)=\frac{\varepsilon_k^2}{\bigl(z-(1+\varepsilon_k)e^{ia_k}\bigr)^3},\qquad\text{and}\qquad
f(z)=\sum_{k=1}^N u_k(z).
\end{equation}

For every $e^{i\theta}\in \T$ and $1\le k\le N$, define
$$
\varphi_k(e^{i\theta})=\int_0^1 |u_k(re^{i\theta})|\, dr.
$$
We will use the following estimate about $\varphi_m$ to be proved later.

\noindent {\bf Claim.} Let $\theta\in\R$ such that $|\theta-a_k|\le \pi$. Then
$$
\varphi_k(e^{i\theta})\le \min\Bigl\{ 1, \frac{8\varepsilon_k^2}{|\theta-a_k|^2}\Bigr\}.
$$
We prove this claim at the end of the proof. 

Observe that, if $(r,e^{i\theta})\in D_k$, then
$$
|re^{i\theta}-(1+\varepsilon_k)e^{ia_k}|\le r|e^{i\theta}-e^{ia_k}| + (1+\varepsilon_k)-r\le
r\delta\varepsilon_k + (1+\varepsilon_k)-(1-\varepsilon_k-\delta\varepsilon_k) 
\le 2(1+\delta)\varepsilon_k < 3\varepsilon_k,
$$
consequently
$$
|u_k(re^{i\theta})|\ge \frac{\varepsilon_k^2}{(3\varepsilon_k)^3}=\frac{1}{27 \varepsilon_k},
$$
and, by Lemma \ref{Bcdelta} with $\varepsilon_k$ in the place of $\eta$,
$$
|u_k(re^{i\theta})||g'(re^{i\theta})|(1-r)\ge \frac{1}{27 \varepsilon_k}\frac{c}{\varepsilon_k}(1-\delta)\varepsilon_k
\ge \frac{c(1-\delta)}{27 \varepsilon_k},
$$
and
\begin{equation}\label{ukgrande}
\int_{D_k} |R_g u_k|\, dm_2\ge m_2(D_k) \frac{c(1-\delta)}{27 \varepsilon_k}=
\frac{(2\delta\varepsilon_k)^2c(1-\delta)}{27 \varepsilon_k}\ge\frac{c \delta^2 \varepsilon_k}{14}.
\end{equation}

Therefore, for every $k$, we have, since $|R_g h(r,e^{i\theta})|\le B|h(re^{i\theta})|$, 
$$
\int_{D_k} |R_g f|\, dm_2 \ge \int_{D_k} |R_g u_k|-\sum_{j\ne k} \int_{D_k} |R_gu_j|\, dm_2
\ge \frac{c\delta^2}{14} \varepsilon_k - B \sum_{j\ne k}\int_{L_k} \varphi_j(e^{i\theta})\, dm(e^{i\theta}).
$$
As $L_k\subset J_k$ and the $J_k$'s are pairwise disjoint, so are the $L_k$'s and the $D_k$'s.
Hence, 
adding up these inequalities from $k=1$ to $k=N$, using \eqref{covering}
and taking into account that $L_k\subset \T\setminus J_j$, for $k\ne j$, we get
\begin{align*}
 \int_D |R_g f|\, dm_2 =& \sum_{k=1}^N \int_{D_k} |R_g f|\, dm_2\ge  
 \frac{c\delta^2}{14}  \sum_{k=1}^N \varepsilon_k - 
B \sum_{j=1}^N \sum_{k=1, k\ne j}^N\int_{L_k} \varphi_j(e^{i\theta})\, dm(e^{i\theta}) \\
 \ge&  \frac{c\delta^2}{14}\sum_{k=1}^N \varepsilon_k- 
B \sum_{j=1}^N \int_{\T\setminus J_j} \varphi_j(e^{i\theta})\, dm(e^{i\theta}).
\end{align*}
By the Claim, we have
$$
\int_{\T\setminus J_j}\varphi_j(e^{i\theta})\, dm(e^{i\theta}) = 
\int_{M\varepsilon_j<|\theta-a_j|<\pi} \varphi_j(e^{i\theta})\,d\theta\le
2\int_{M\varepsilon_j}^\pi \frac{8\varepsilon_j^2}{t^2}\, dt\le 
16\varepsilon_j^2\int_{M\varepsilon_j}^{+\infty} \frac{dt}{t^2}
=\frac{16\varepsilon_j^2}{M\varepsilon_j}.
$$
Putting together the last two estimates, if $M>\frac{2\times16\times 14}{c\delta^2} \, B$, we obtain
$$
 \int_D |R_g f|\, dm_2 \ge \Bigl(\frac{c\delta^2}{14}-\frac{16 B}{M}\Bigr)\sum_{k=1}^N\varepsilon_k
 \ge \frac{c\delta^2}{28}\sum_{k=1}^N \varepsilon_k,
$$
and, by \eqref{sumaes},
$$
\int_D |R_g f|\, dm_2\ge \frac{c\delta^2}{28} \frac{\beta}{6M} = \frac{c\delta^2\beta}{168 M}:=\a.
$$
We have established \eqref{noequiint} (c).

Now we prove the bound for $\rho_{1,q}(f)$. Naturally we have
$$
\rho_{1,q}(f)\le \biggl(\frac{1}{2\pi} \int_{-\pi}^{\pi}
 \Bigl(\sum_{k=1}^N\varphi_k(e^{i\theta})\Bigr)^q   \, d\theta \biggr)^{1/q}.
$$
In order to apply the estimate in the Claim, the condition  
$|\theta-a_k|\le \pi$ has to be satisfied. Let us assume that all the $a_k$'s belong to 
the interval $[0,2\pi)$, then, for all $\theta\in (-\pi,\pi]$, either
$|\theta-a_k|\le \pi$ or $|\theta-(a_k-2\pi)|\le \pi$. Define 
$$
I_k^+:=[a_k-\varepsilon_k,a_k+\varepsilon_k]\qquad \text{and}\qquad
I_k^-:=[a_k-2\pi -\varepsilon_k,a_k-2\pi +\varepsilon_k].
$$
Observe that, if $g$ is the characteristic function of the
interval $[a-\varepsilon,a+\varepsilon]$ and $\mathcal{M}g$ is its Hardy-Littlewood maximal
function, we have
$\mathcal{M}g(t)=1$, if $t\in (a-\varepsilon,a+\varepsilon)$, 
and, for $t\notin (a-\varepsilon,a+\varepsilon)$,
$$
\mathcal{M}g(t)\ge \frac{1}{2|t-a|}\int_{t-|t-a|}^{t+|t-a|} g(t)\, dt=\frac{\varepsilon}{2|t-a|}\,.
$$
Now if  $g_k^+$ is the characteristic function of $I_k^+$ and 
$g_k^-$ is the characteristic function of $I_k^-$, 
using the Claim, we have, for every $\theta\in (-\pi,\pi]$,
$$
\varphi_k(e^{i\theta})\le 32 \bigl[\bigl(\mathcal{M}g_k^+\bigr)^2(\theta)
+ \bigl(\mathcal{M}g_k^-\bigr)^2(\theta)     \bigr].
$$
Therefore, 
$$
\rho_{1,q}(f)\le 32\biggl(\int_\R \Bigl(\sum_{k=1}^N
\bigl[\bigl(\mathcal{M}g_k^+\bigr)^2(t)
+ \bigl(\mathcal{M}g_k^-\bigr)^2(t)     \bigr]\Bigr)^q \,dt\biggr)^{1/q}
=32 \|H\|_{L^{2q}(\R)}^2,
$$
where 
$$
H=\Bigl(\sum_{k=1}^{N} \bigl[\bigl(\mathcal{M}g_k^+\bigr)^2 + \bigl(\mathcal{M}g_k^-\bigr)^2 \bigr]\Bigr)^{1/2}.
$$
Applying \cite[Theorem 1]{fefferman_stein_1971}, there exists a constant $A_{2,2q}>0$ such that
$$
\|H\|_{L^{2q}(\R)}\le A_{2,2q} \|h\|_{L^{2q}(\R)},\qquad \text{ for }\ 
h=\bigl(\sum_k [(g_k^+)^2 + (g_k^-)^2 ]\bigr)^{1/2}.
$$
Since all the intervals $I_k^+$'s and $I_k^-$'s are pairwise disjoint, we see easily that
$h$ is the characteristic of the union of all these intervals and so $\|h\|_{2q}\le (4\pi)^{1/2q}$.
Finally we get
$$
\rho_{1,q}(f) \le 32 \|H\|_{L^{2q}(\R)}^2\le 32 A_{2,2q}^2\|h\|_{L^{2q}(\R)}^2 
\le 32 A_{2,2q}^2 (4\pi)^{1/2q} =: C,
$$
and we finish because we have proved \eqref{noequiint} (b).

\noindent {\bf Proof of the Claim.} 
By rotation invariance, we can assume $a_k=0$. Then, for all $\theta\in [-\pi,\pi]$ and
all $r\in [0,1]$, we have $|re^{i\theta} - (1+\varepsilon_k)|\ge |(1+\varepsilon_k)-r|$.
This yields $|u_k(re^{i\theta})|\le |u_k(r)|$, and
\begin{equation}\label{menorque1}
\varphi_k(e^{i\theta})\le \varphi_k(e^{i0})=\int_0^1\frac{\varepsilon_k^2}{(1+\varepsilon_k -r)^3}\,dr
=\Bigl(\frac{\varepsilon_k^2}{ 2(1+\varepsilon_k -r)^2}\Bigr]_{r=0}^{r=1}\le \frac{1}{2} \le 1.
\end{equation}
On the other side, for all $z=re^{i\theta}\in\D$, we have
\begin{equation}\label{doscasos}
|1+\varepsilon_k-z|\ge |1-z|\ge 
\begin{cases}
 |\sin \theta|\ge 2|\theta|/\pi  ,&\quad \text{if }\  0<|\theta|< \pi/2 , \\
 1,&\quad\text{if }\ \pi/2\le |\theta|\le \pi .
\end{cases}
\end{equation}
Therefore, if $\pi/2\le |\theta|\le \pi$, we have 
\begin{equation*}
|1-re^{i\theta}|^3\ge | 1-re^{i\theta}|^2=1+r^2-2r\cos\theta\ge 1+r^2\qquad\text{and}
\qquad |u_k(re^{i\theta})|\le \frac{\varepsilon_k^2}{1+r^2}.
\end{equation*}
Integrating
\begin{equation}\label{alpi}
\varphi_k(e^{i\theta})\le \varepsilon_k^2 \int_0^1 \frac{dr}{1+r^2}=\frac{\pi\varepsilon_k^2}{4}\le
\frac{\pi^3}{32}\frac{8 \varepsilon_k^2}{|\theta|^2}\le \frac{8 \varepsilon_k^2}{|\theta|^2},
\qquad\text{if }\ \frac{\pi}{2}\le |\theta|\le \pi.
\end{equation}
For $1\le |\theta|\le \pi/2$, we use the first case in \eqref{doscasos}. We have
\begin{equation}\label{alpimedio}
\varphi_k(e^{i\theta})\le \varepsilon_k^2 \Bigl(\frac{\pi}{2|\theta|}\Bigr)^3
\le \frac{\pi^3\varepsilon_k^2}{8|\theta|^2}=
\frac{\pi^3}{64}\frac{8 \varepsilon_k^2}{|\theta|^2}\le \frac{8 \varepsilon_k^2}{|\theta|^2},
\qquad\text{if }\ 1\le  |\theta| \le \frac{\pi}{2}.
\end{equation}

Finally, for $|\theta|<1$, we have
\begin{equation*}
\varphi_k(e^{i\theta})\le\varepsilon_k^2\int_0^1 \frac{dr}{|1-re^{i\theta}|^3}\le 
\varepsilon_k^2\int_0^{1-|\theta|}\frac{dr}{(1-r)^3}+\varepsilon_k^2\int_{1-|\theta|}^1
\frac{\pi^3}{8|\theta|^3}\,dr,
\end{equation*}
and we get
\begin{equation}\label{yavahostia}
\varphi_k(e^{i\theta})\le \varepsilon_k^2\Bigl(\frac{1}{2|\theta|^2} + \frac{\pi^3}{8|\theta|^2} \Bigr)
\le\Bigl(\frac{1}{2}+\frac{\pi^3}{8}\Bigr)\frac{ \varepsilon_k^2}{|\theta|^2}
\le \frac{ 8\varepsilon_k^2}{|\theta|^2},
\qquad\text{if }\ 0\le  |\theta| <1.
\end{equation}
Putting together \eqref{menorque1}, \eqref{alpi},  \eqref{alpimedio},  and  
\eqref{yavahostia}, the lemma follows.
\end{proof}

\begin{theorem}\label{Thm:weakly-compact-RM(1,q)}
	Let $1<q<+\infty$ and $g\in \mathcal{B}$. Then  $T_g:RM(1,q)\rightarrow RM(1,q)$ is weakly compact if and only if $g\in \mathcal{B}_{0,w}$. 
\end{theorem}

\begin{proof} Assume that $g\in \mathcal{B}_{0,w}$. 
For each $\varepsilon,\delta >0$, we  set
	\begin{align*}
	A(\delta,\varepsilon):=\left\{\theta \in \mathbb{T}: \ (1-r) |g'(re^{i\theta})|<\varepsilon,\quad {\textrm{ for all } } r\in (1-\delta,1) \right\}.
	\end{align*}
Fixed $m\in \N$. By hypothesis, we have that 
	\begin{align*}
	m\left(A\left(\frac{1}{n},\frac{1}{2^m}\right)\right)\rightarrow 1
	\end{align*}
	for $n\rightarrow \infty$. Moreover, it can be seen that $A\left(\frac{1}{n},\frac{1}{2^m}\right)\subset A\left(\frac{1}{n+1},\frac{1}{2^m}\right)$.
	Hence, for each $m\in\N$ there is $n_m\in \N$ such that 
	\begin{align*}
	m\left(A\left(\frac{1}{n_m},\frac{1}{2^m}\right)\right)>1-\frac{1}{m^2}.
	\end{align*}
	So, we have that
	\begin{equation}\label{Eq:sufficient-weakly-compact}
	\lim\limits_{k\rightarrow \infty } m\left(\bigcap_{m\geq k} A\left(\frac{1}{n_m},\frac{1}{2^m}\right)\right)=1.
	\end{equation}
	
Fix $\varepsilon>0$, by \eqref{Eq:sufficient-weakly-compact}, there is $k=k(\varepsilon)$ such that $m(A_{\varepsilon})>1-\varepsilon$ where
\begin{equation}\label{Eq:sufficient-weakly-compact2}
A_{\varepsilon}:=\bigcap_{m\geq k} A\left(\frac{1}{n_m},\frac{1}{2^m}\right).
\end{equation}
This means that given $\theta \in A_{\varepsilon}$, for each $m\geq k$, 
$$
(1-r) |g'(re^{i\theta})|<1/2^{m}
$$
whenever $1-1/n_{m}<r<1$. 	

To obtain the weak compactness, we apply  Lemma~\ref{lemmaweakcompact1q}.   
Let $\{f_n\}\in B_{RM(1,q)}$ be a sequence uniformly convergent to $0$ on compact sets. Define the functions
\begin{align*}
	H_{n}(\theta):=\int_{0}^{1} |f_{n}(re^{i\theta})g'(re^{i\theta})|(1-r)\ dr \quad \textrm{ and }\quad  F_{n}(\theta):=\int_{0}^{1} |f_{n}(re^{i\theta})|\ dr.
\end{align*}
Using that $T_g$ is bounded on $RM(1,q)$ and Proposition  \ref{bergtypine}, the sequence $\{H_{n}\}$ is bounded on $L^q(\mathbb{T})$. 
Then, by the reflexivity of  this space, we can find a subsequence $\{H_{n_k}\}$ convergent in the weak topology to a function $h\in L^q(\mathbb{T})$. Therefore, there is $G_{k}\in \textrm{co}\{H_{n_k},H_{n_{k+1}},\dots\}$ such that $\|G_{k}-h\|_{L^q(\mathbb{T})}\rightarrow 0$.
We claim that $h=0$. To  settle this fact, fix $\varepsilon >0$. By \eqref{Eq:sufficient-weakly-compact2}, there are $N=N(\varepsilon)\in \N$ and  a measurable set $A_{\varepsilon}$ with $m(A_{\varepsilon})>1-\varepsilon$ and for every $\theta \in A_{\varepsilon}$ and  $m\geq N$, 
$$
(1-r) |g'(re^{i\theta})|<1/2^{m}
$$
whenever $1-1/n_{m}<r<1$. 
We may assume that $1/2^{N}<\epsilon$ and that for $m\geq N$ and $r<1-1/n_{N}$,
$$
|f_{m}(re^{i\theta})g'(re^{i\theta})|(1-r)\leq \varepsilon
$$
(remember that the sequence $\{f_{n}\}$ converge uniformly to  $0$ on the disc center at $0$ and radius $1-1/2^{N}$). 
Thus, for $n\geq N$ and $\theta\in A_\varepsilon$,
\begin{equation}\label{ineqAset}
\begin{split}
H_{n}(\theta)&=\int_{0}^{1-1/n_{N}} |f_{n}(re^{i\theta})g'(re^{i\theta})|(1-r)\ dr+ \int_{1-1/n_{N}}^{1} |f_{n}(re^{i\theta})g'(re^{i\theta})|(1-r)\ dr\\
&\leq \varepsilon+\varepsilon F_{n}(\theta)
\end{split}
\end{equation}
So, it follows that $\|H_n \chi_{A_\varepsilon}\|_{L^q (\mathbb{T})}\leq 2\varepsilon$ for all $n>N$. Hence, $\|G_n \chi_{A_\varepsilon} \|_{L^q(\mathbb{T})}<3\varepsilon $ for $n$ large enough. This implies  that $h\chi_{A_\varepsilon} =0$. The arbitrariness of $\varepsilon $ and the fact that $m(A_{\varepsilon})>1-\varepsilon$ implies that $h=0$ and $\|G_{k}\|_{L^q(\mathbb{T})}\rightarrow 0$.

 Notice that we can express $G_{k}$ in the following way
	\begin{align*}
	G_{k}=\sum_{j=k}^{\infty}\alpha_{k,j} H_{n_j}
	\end{align*}
	where $\alpha_{k,j}\geq 0$, $\sum_{j=k}^{\infty} \alpha_{k,j}=1$ and, for each $k$, the set $\{j\geq k:\, \alpha_{k,j}\neq 0\}$ is finite. Thus the functions
	\begin{align*}
	g_{k}:=\sum_{j=n}^{\infty}\alpha_{k,j} f_{n_j},
	\end{align*}
	are well-defined and it follows that 
	\begin{align*}
	\int_{0}^{1} |T_{g} g_{k}|\ dr\leq \sum_{j=k}^{\infty}\alpha_{k,j} H_{n_j}=G_{k}(\theta).
	\end{align*}
	Hence $\rho_{1,q}(T_{g}g_k)\rightarrow 0$ when $k\rightarrow \infty$. Therefore, using Lemma~\ref{lemmaweakcompact1q} we conclude that $T_g:RM(1,q)\rightarrow RM(1,q)$ is weakly compact.
	
For the converse implication assume that $g\in \mathcal B\setminus \mathcal B_{0,w}$. By Propositions \ref{bergtypine} and \ref{converse-little-paley}, given $f\in RM(1,q)$, it holds that $\rho_{1,q}(T_{g}f)\asymp \Vert R_{g}f\Vert _{L^{q}(\T, L^{1}([0,1]))}$ where $R_{g}$ is the operator introduced in Theorem \ref{technical result} with the identification $L^{q}(\T, L^{1}([0,1]))\subset L^{1}(\T, L^{1}([0,1]))= L^{1}([0,1)\times \T)$. Therefore, $R_{g}:RM(1,q) \to L^{1}([0,1)\times \T)$ is bounded and not weakly compact. 
By Lemma \ref{lemmaweakcompact1q}, there exists a sequence $\{f_n\}$ in the unit ball of $RM(1,q)$  convergent to $0$ uniformly on compact sets of $\D$ such that no convex combination $g_k\in \mathrm{co}\{f_k,f_{k+1},\dots\}$ satisfies that $ \Vert R_{g}g_{k}\Vert _{L^{1}(\T, L^{1}([0,1]))  }\rightarrow 0$ when $k\rightarrow \infty$. Applying again Propositions \ref{converse-little-paley},  no convex combination $g_k\in \mathrm{co}\{f_k,f_{k+1},\dots\}$ satisfies that $ \rho_{1,1}(T_{g}g_{k})\rightarrow 0$ when $k\rightarrow \infty$. By Lemma \ref{lemmaweakcompact1q}, $T_{g}:RM(1,q)\to RM(1,1)$ is not weakly compact and, by Proposition \ref{Prop: weak-compact-1-q}, $T_{g}:RM(1,q)\to RM(1,q)$ is not weakly compact.
\end{proof}

We point out that beyond what it seems in Proposition \ref{Prop: weak-compact-1-q}, the weak compactness $T_g:RM(1,q)\rightarrow RM(1,q)$ does not depend on $q$ when it runs the interval $q\in (1,+\infty)$.

\begin{remark} Using \cite[Proposition 5.4, p. 601]{ramey_bounded_1991}, there are $g_1,g_2\in\mathcal{B}$ such that
	\begin{align*}
	|g'_1(z)|+|g'_2(z)|\geq \frac{1}{1-|z|}, \quad \textrm{ for all } z\in \D.
	\end{align*}
Therefore either $g_{1}$ or $g_{2}$ does not belong to $\mathcal{B}_{0,w}$, so that $\mathcal{B} \setminus \mathcal{B}_{0,w}$ is not empty.  
Moreover the function $g(z)=\log(1-z)$, $z\in \D$, belongs to $\mathcal{B}_{0,w}\setminus\mathcal{B}_{0}$. In fact, writing $g_{\theta}(z)=\log(1-ze^{i\theta})$ for $\theta\in[0,\pi)$ and $z\in \D$, one can see that $\|g_{\theta}-g_{\tilde{\theta}}\|_{\mathcal{B}}\geq 1$ if $\theta\neq \tilde{\theta}$. Then $\mathcal{B}_{0,w}$ is a non-separable closed  subspace of $\mathcal{B}$. The separability of  $\mathcal{B}_{0}$ and the non-separability of  $\mathcal{B}_{0,w}$ show that the second one is much bigger than the first one. Therefore there are integral operators $T_{g}$ bounded and not weekly compact on $RM(1,q)$ and integral operators $T_{g}$ weakly compact and not compact. 
\end{remark}

\begin{remark}
By \cite[Proposition 4.8]{Pombook92}, if $g\in \mathcal H(\D)$ and $\Im g$ has a finite angular limit at $e^{it}$, then $(z- e^{it})g'(z)$ has angular limit $0$ at $e^{it}$. 
This result implies that if $g\in \mathcal B$ and    $\Im g$ has a finite angular limit at $e^{it}$ for almost every $e^{it}\in \T$, then $g\in \mathcal B_{0,w}$. In particular, $H^{p}\subset \mathcal B_{0,w}$. This last inclusion was firstly noticed by Pavlovi\'c \cite[Corollary, 2.1]{Pavlovic}. 
\end{remark}

\subsection{The case $q=1$} To finish, we turn our attention to the weak compactness of $T_{g}: RM(p,1)\rightarrow RM(p,1)$. The following three lemmas will be necessary to give a characterization of the weak compactness of $T_g$ by means of sequences in $(RM(p,1))^\ast$ which are equivalent to the basis of $c_0$.

\begin{lemma}\label{deltabasisc0}
Let $1\leq p\leq +\infty$ and 	let $\{z_n\}$ be such that there are constants $C_1,C_2,C_3>0$ satisfying:
	\begin{enumerate}
		\item $\sum_{n=1}^{\infty} |f(z_n)|(1-|z_n|)^{1+\frac{1}{p}}\leq C_1 \rho_{p,1}(f)$ for all $f\in RM(p,1)$.
		\item For all $m\in \N$, there is $f_m\in RM(p,1)$ with $\rho_{p,1}(f_m)\leq C_2$ such that
		\begin{align*}
		|f_{m}(z_m)|(1-|z_m|)^{1+\frac{1}{p}}\geq \frac{1}{C_3}, \qquad
		\sum_{n\neq m} |f_{m}(z_n)|(1-|z_n|)^{1+\frac{1}{p}} \leq \frac{1}{2C_3}.
		\end{align*}
	\end{enumerate}
	Then $\{ (1-|z_n|)^{1+\frac{1}{p}}\ \delta_{z_n}\}$ is equivalent to the basis of $c_0$ in $(RM(p,1))^{\ast}$.
\end{lemma}

\begin{proof} We will present the proof for $p$ finite, being the other case similar. It is sufficient to prove that there are constants $A>0$ and $B>0$ such that
	\begin{align*}
	A \max_{1\leq k\leq N}\{|\alpha_{k}|\}\leq \left\|\sum_{k=1}^{N} \alpha_{k} (1-|z_k|)^{1+\frac{1}{p}}\delta_{z_k}\right\|_{(RM(p,1))^{\ast}} \leq B \max_{1\leq k\leq N}\{|\alpha_{k}|\}
	\end{align*}
	for every $N$ and for every sequence $\{\alpha_{k}\}$.
	
	First, using assertion (1) and Proposition \ref{main-properties}, we have that
	\begin{align*}
	&\left\|\sum_{k=1}^{N} \alpha_{k} (1-|z_k|)^{1+\frac{1}{p}}\delta_{z_k}\right\|_{(RM(p,1))^{\ast}} =\sup_{f\in B_{RM(p,1)}} \left|\sum_{k=1}^{N} \alpha_{k} (1-|z_k|)^{1+\frac{1}{p}}f(z_k)\right|\\
	&\leq \max_{1\leq k\leq N}\{|\alpha_{k}|\} \sup_{f\in B_{RM(p,1)}}\sum_{k=1}^{\infty} |f(z_k)|(1-|z_k|)^{1+\frac{1}{p}} \leq C_1 \max_{1\leq k\leq N}\{|\alpha_{k}|\}.
	\end{align*}
	
	The remaining inequality proceeds as follows employing this time assertion (2). We choose $m$ such that $\max_{1\leq k\leq N}\{|\alpha_{k}|\}=|\alpha_m|$. Then, we obtain that
	\begin{align*}
	&\left\|\sum_{k=1}^{N} \alpha_{k} (1-|z_k|)^{1+\frac{1}{p}}\delta_{z_k}\right\|_{(RM(p,1))^{\ast}}\geq \frac{1}{C_2} \left|\sum_{k=1}^{N} \alpha_{k} (1-|z_k|)^{1+\frac{1}{p}}f_{m}(z_k)\right|\\
	&\geq \frac{|\alpha_m|}{C_2} (1-|z_m|)^{1+\frac{1}{p}}|f_{m}(z_m)|-\frac{1}{C_2}\sum_{k=1,\, k\neq m}^{N} |\alpha_{k}|(1-|z_k|)^{1+\frac{1}{p}}|f_{m}(z_k)|\\
	&\geq \frac{1}{2C_2C_3}\max_{1\leq k\leq N}\{|\alpha_{k}|\}.
	\end{align*}
\end{proof}

\begin{lemma}\label{meanlittwind}
	Let $c\in (0,1/2)$, then there are two constants $\mu_1,\mu_2$, depending only on $c$, such that for all $ f\in \mathcal{H}(\D)$ 
	\begin{align*}
	|f(z)|&\leq \frac{\mu_1}{(1-|z|)^2}\int_{\theta-c(1-r)}^{\theta+c(1-r)} \left(\int_{r-c(1-r)}^{r+c(1-r)}|f(\rho e^{i t})|\ d\rho\right)\ dt, \\
	(1-|z|)|f'(z)|&\leq \frac{\mu_2}{(1-|z|)^2}\int_{\theta-c(1-r)}^{\theta+c(1-r)} \left(\int_{r-c(1-r)}^{r+c(1-r)}|f(\rho e^{i t})|\ d\rho\right)\ dt,
	\end{align*} 
	where $z=re^{i\theta}$ with $1>r\geq \frac{1}{2}$.
\end{lemma}

\begin{proof}
	We will prove the first inequality, since the proof of the last one is analogous. 
	
	Let $c\in (0,1/2)$, $z=re^{i\theta}\in \D$ and $f\in \mathcal{H}(\D)$. It can be proved that for $\lambda \in \left(0,\frac{1}{\sqrt{4+c^2}}\right)$ we have that
	\begin{align*}
	D(z,\lambda c (1-|z|))\subset \left\{ \rho e^{i t}\in\D :\ \rho\in I, t\in J\right\}
	\end{align*}
	where $I=[r-c(1-r),r+c(1-r)]$ and $J=[\theta-c(1-r),\theta+c(1-r)]$.

	Applying the mean value inequality over $D(z,\lambda c (1-|z|))$ we have that
	\begin{align*}
	|f(z)|&\leq \frac{1}{\pi\lambda^2 c^2 (1-r)^2} \int_{D(z,\lambda c (1-|z|))} |f(w)|\ dw\\
	&\leq \frac{1}{\pi\lambda^2 c^2 (1-r)^2} \int_{\theta-c(1-r)}^{\theta+c(1-r)} \left(\int_{r-c(1-r)}^{r+c(1-r)}|f(\rho e^{i t}|\ d\rho\right)\ dt.
	\end{align*}
\end{proof}

\begin{lemma}\label{sequencedeltaequibasisc0}
	Let $1\leq p<+\infty$ and $\{z_k\}\subset \D\setminus\frac{1}{2}\D$ such that $\frac{1-|z_{n+1}|}{1-|z_n|}\leq \frac{\beta}{n^2}$ with $\beta \in \left(0,\frac{1}{1+2^{4+\frac{1}{p}}}\right)$. Then both $\{ (1-|z_n|)^{1+\frac{1}{p}}\ \delta_{z_n}\}$ and $\{ (1-|z_n|)^{2+\frac{1}{p}}\ \delta'_{z_n}\}$ are equivalent to the basis of $c_0$ in $(RM(p,1))^{\ast}$.
\end{lemma}

\begin{proof}
	We will prove the result just for $\{ (1-|z_n|)^{1+\frac{1}{p}}\ \delta_{z_n}\}$ using Lemma~\ref{deltabasisc0} and omit the proof of the other case because it can be obtained following a similar argument.
	Set $z_n=r_{n} e^{i\theta_n}$ and $\varepsilon_n=1-r_n$. For a certain constant $c\in (0,1/2)$ we define the sets
	$$I_{n}:=[\theta_n-c(1-r_n),\theta_n+c(1-r_n)]$$ and $$J_{n}:=[r_n-c(1-r_n),r_n+c(1-r_n)].$$ Now, we denote by $A_{n}:=\cup_{k>n} I_k$ where $m(A_n)\leq \sum_{k=n+1}^{\infty} 2c\varepsilon_{k}\leq \frac{1}{n^2} \varepsilon_n$.
	
	Let $f\in B_{RM(p,1)}$. Applying Lemma~\ref{meanlittwind} to each element of the sequence $\{z_n\}$, we obtain that 
	\begin{align*}
	(1-|z_n|)^{1+\frac{1}{p}} |f(z_n)| \leq \frac{\mu_1}{\varepsilon_{n}^{1/p'}} \int_{I_n} \int_{J_n} |f(\rho e^{it})|\ d\rho\,  dt.
	\end{align*}
	We split the integrals,  apply H\"older's inequality in the first integral and  Proposition \ref{main-properties}(1)  in the second integral:
	\begin{align*}
	(1-|z_n|)&^{1+\frac{1}{p}} |f(z_n)| \leq \frac{\mu_1}{\varepsilon_{n}^{1/p'}} \int_{I_n\setminus A_n} \int_{J_n} |f(\rho e^{it})|\, d\rho dt+\frac{\mu_1}{\varepsilon_{n}^{1/p'}} \int_{A_n} \int_{J_n} |f(\rho e^{it})|\ d\rho \, dt\\
	&\leq \frac{\mu_1 (2c\varepsilon_{n})^{1/p'}}{\varepsilon_{n}^{1/p'}} \int_{I_n\setminus A_n} \left(\int_{J_n} |f(\rho e^{i t})|^{p}\ d\rho\right)^{1/p}\, dt +\frac{C\mu_1 m(A_n) m(J_n)}{\varepsilon_{n}^{1/p'} (1-(r_n+c(1-r_n)))^{1+\frac{1}{p}}} \\
	&\leq \mu_1 (2c)^{1/p'}\int_{I_n\setminus A_n} \left(\int_{J_n} |f(\rho e^{i t})|^{p}\ d\rho\right)^{1/p}\, dt + \frac{2Cc \mu_1 }{n^2 (1-c)^{1+\frac{1}{p}}}.
	\end{align*}
	Since  $\{I_n\setminus A_n\}$ are disjoint sets, we have that
	\begin{align*}
	\sum_{n=1}^{\infty}(1-|z_n|)^{1+\frac{1}{p}} |f(z_n)|\leq \mu_1 (2c)^{1/p'}+\frac{\pi^2 C c \mu_1}{3 (1-c)^{1+\frac{1}{p}}}=C_1.
	\end{align*}
	Hence, we have proved that the sequence $\{z_{n}\}$ satisfies statement (1) of Lemma~\ref{deltabasisc0}. To prove the remaining condition we consider the family of holomorphic functions $f_n(z):=\frac{1-r_n}{(1-\overline{z_n}z)^{2+\frac{1}{p}}}$, $z\in \D$. We have to show that  $\{f_{n}\}$ satisfies statement (2) of Lemma~\ref{deltabasisc0}. 
	
	Let us see that $\rho_{p,1}(f_n)\lesssim 1$. First of all, we observe that 
	\begin{align*}
	\rho_{p,1}(f_n)&=\int_{0}^{2\pi}\left(\int_{0}^{1} \frac{(1-r_n)^{p}}{|1-\overline{z_n}re^{i\theta}|^{2p+1}}\ dr\right)^{1/p}\ \frac{d\theta}{2\pi}\\
	&\leq 8(1-r_n)\int_{0}^{\pi/4}\left(\int_{0}^{1} \frac{1}{|1-r r_{n} e^{i\theta}|^{2p+1}}\ dr\right)^{1/p}\ \frac{d\theta}{2\pi}.
	\end{align*}
	
Since $|1-rr_{n}e^{i\theta}|\geq \frac{1}{4}\theta$ whenever $0\leq \theta\leq \pi/4$ and $1\geq r\geq \frac{1-\theta}{r_{n}}$, we have	
	\begin{align*}
	&\int_{0}^{\pi/4}\left(\int_{0}^{1} \frac{1}{|1-r r_{n} e^{i\theta}|^{2p+1}}\ dr\right)^{1/p}\ \frac{d\theta}{2\pi}\leq  \int_{0}^{1-r_n}\left(\int_{0}^{1} \frac{1}{|1-r r_{n} e^{i\theta}|^{2p+1}}\ dr\right)^{1/p}\ \frac{d\theta}{2\pi}\\
	&\qquad + \int_{1-r_n}^{\pi/4}\left(\int_{0}^{\frac{1-\theta}{r_n}} \frac{1}{|1-r r_{n} e^{i\theta}|^{2p+1}}\ dr+\int_{\frac{1-\theta}{r_n}}^{1} \frac{1}{|1-r r_{n} e^{i\theta}|^{2p+1}}\ dr\right)^{1/p}\ \frac{d\theta}{2\pi} \\
&\leq	\int_{0}^{1-r_n}\left(\int_{0}^{1} \frac{1}{(1-r r_{n})^{2p+1}}\ dr\right)^{1/p}\ \frac{d\theta}{2\pi}\\
	&\qquad 
	+ \int_{1-r_n}^{\pi/4}\left(\int_{0}^{\frac{1-\theta}{r_n}} \frac{1}{(1-r r_{n})^{2p+1}}\ dr+4^{2p+1} \int_{\frac{1-\theta}{r_n}}^{1} \frac{1}{\theta^{2p+1}}\ dr\right)^{1/p}\ \frac{d\theta}{2\pi}\\
	&\leq \frac{1-r_n}{2\pi(2pr_n)^{1/p}} \left(\frac{1}{(1-r_n)^{2p}}\right)^{1/p} + 4^{3}\int_{1-r_n}^{\pi/4}\left( \frac{1}{2pr_n}\left(\frac{1}{\theta^{2p}}\right)+\frac{1}{\theta^{2p+1}}\left(1-\frac{1-\theta}{r_n}\right)\right)^{1/p}\ \frac{d\theta}{2\pi}\\
	&\leq \frac{1}{2\pi(2pr_n)^{1/p}(1-r_n)} + 4^{3}\frac{(2p+1)^{1/p}}{2\pi(2pr_n)^{1/p}}\int_{1-r_n}^{\pi/4} \frac{1}{\theta^2}\ \frac{d\theta}{2\pi}\\
	&\leq \frac{1}{2\pi(2pr_n)^{1/p}(1-r_n)} + 4^{3}\frac{(2p+1)^{1/p}}{2\pi(2pr_n)^{1/p}(1-r_n)}\leq 4^{3}\frac{(2p+1)^{1/p}+1}{2\pi(2pr_n)^{1/p}(1-r_n)}.
	\end{align*}
	Therefore, we conclude that 
	\begin{align*}
	\rho_{p,1}(f_n)\leq 4^{3}\frac{(2p+1)^{1/p}+1}{p^{1/p}}=C_2. 
	\end{align*}
	
	To finish the proof, we have to show that for a certain constant $C_3>0$ it holds
	\begin{align*}
	|f_{m}(z_m)|(1-|z_m|)^{1+\frac{1}{p}}\geq \frac{1}{C_3}, \qquad \sum_{n\neq m} |f_{m}(z_n)|(1-|z_n|)^{1+\frac{1}{p}} \leq \frac{1}{2C_3}.
	\end{align*}
	It is easy to see that
	\begin{align*}
	|f_m(z_m)|(1-|z_m|)^{1+\frac{1}{p}}=\frac{1-r_m}{(1-r_m^2)^{2+\frac{1}{p}}}(1-r_m)^{1+\frac{1}{p}}=\frac{1}{(1+r_m)^{2+\frac{1}{p}}}\geq \frac{1}{2^{2+\frac{1}{p}}}=:\frac{1}{C_3}>0.
	\end{align*}
	If $k>m$ then 
	\begin{align*}
	|f_{m}(z_k)|(1-|z_k|)^{1+\frac{1}{p}}\leq \frac{1-r_m}{(1-r_m)^{2+\frac{1}{p}}}(1-r_k)^{1+\frac{1}{p}}=\left(\frac{\varepsilon_k}{\varepsilon_m}\right)^{1+\frac{1}{p}}.
	\end{align*}
	And if $k<m$, then
	\begin{align*}
	|f_{m}(z_k)|(1-|z_k|)^{1+\frac{1}{p}}\leq \frac{1-r_m}{(1-r_k)^{2+\frac{1}{p}}}(1-r_k)^{1+\frac{1}{p}}=\frac{\varepsilon_m}{\varepsilon_k}.
	\end{align*}
	Now, bearing in mind that $0<\beta<\frac{1}{1+2^{4+\frac{1}{p}}}$ such that $\frac{\varepsilon_{n+1}}{\varepsilon_n}<\beta$ for all $n$, we obtain that 
	\begin{align*}
	\sum_{k\neq m} |f_{m}(z_k)|(1-|z_k|)^{1+\frac{1}{p}} <\sum_{k=m+1}^{\infty} \beta^{k+\frac{k}{p}}+\sum_{k=1}^{m-1} \beta^k\leq \frac{\beta^{1+\frac{1}{p}}}{1-\beta^{1+\frac{1}{p}}}+\frac{\beta-\beta^m}{1-\beta}<\frac{2\beta}{1-\beta}<\frac{1}{2C_3}.
	\end{align*}
	
	Therefore, applying Lemma~\ref{deltabasisc0} we have proved that $\{(1-|z_n|)^{1+\frac{1}{p}}\ \delta_{z_n} \}$ is equivalent to the basis of $c_0$ in $(RM(p,1))^{\ast}$.
\end{proof}

As consequence of these lemmas we obtain a characterization of the weak compactness of $T_g: RM(p,1)\rightarrow RM(p,1)$.

\begin{theorem} \label{Thm:T*fixes c0}
Let $g\in\mathcal{B}$ and $1\leq p<+\infty$. If $T_g: RM(p,1)\rightarrow RM(p,1)$ is not compact, 
then $T_{g}^{\ast}$ fixes a copy of $c_0$.
\end{theorem}
\begin{proof}
	Since $g\in\mathcal{B}$ and $T_g: RM(p,1)\rightarrow RM(p,1)$ is not compact, by Theorem \ref{thboundedopinteg}, there are constants $C,\delta>0$ and a sequence $\{z_n\}$ where $|z_n|\rightarrow 1$ such that 
	$$\delta\leq |g'(z_n)|(1-|z_n|)\leq C.$$
	Using that 
	$$
	T_{g}^{\ast}((1-|z_n|)^{2+\frac{1}{p}}\delta'_{z_n})=g'(z_n) (1-|z_n|)^{2+\frac{1}{p}}\delta_{z_{n}}, \quad |g'(z_n)| (1-|z_n|)^{2+\frac{1}{p}} \asymp  (1-|z_n|)^{1+\frac{1}{p}} ,
	$$ 
	and extracting a subsequence  such that $\frac{1-|z_{n_{k+1}}|}{1-|z_{n_k}|}\leq \frac{\beta}{k^2}$ with $\beta \in \left(0,\frac{1}{1+2^{4+\frac{1}{p}}}\right)$, by Lemma \ref{sequencedeltaequibasisc0}, we conclude that $T_{g}^{\ast}$ fixes a copy of $c_0$. 
\end{proof}

\begin{corollary} \label{Cor:T*fixes c0} Let $g\in\mathcal{B}$ and $1\leq p<+\infty$. Then the following are equivalent:
\begin{enumerate}
\item $T_g: RM(p,1)\rightarrow RM(p,1)$ is not compact;
\item $T_{g}^{\ast}$ fixes a copy of $c_0$;
\item  $T_{g}$ fixes a copy of $\ell_{1}$.
\end{enumerate}
\end{corollary}
\begin{proof} Being clear that (3) implies (1) and, using Theorem \ref{Thm:T*fixes c0}, that (1) implies (2), we just have to justify that (2) implies (3). But this is a consequence of the fact that if the adjoint of a bounded operator between two Banach spaces fixes a copy of $c_{0}$, then the operator fixes a copy of $\ell_{1}$. This result is probably well-known by specialist (and essentially due to C. Bessaga and A. Pe\l czy\'nski), but  we could not find any reference so that we schedule its proof for the sake of completeness. Assume that  $T:X\to Y$ is bounded and $T^{*}$ is fixes a copy of $c_{0}$.  Then $T^{*}$ is unconditionally converging (\cite{Bessaga-Pelczynski}, \cite[Exercise 8, page 54]{Diestel}), so that $T$ is not an $\ell_{1}$-cosingular operator (\cite{Pelczynski},  \cite[page 273]{Howard}). But a standard argument shows that in this case $T$ fixes a copy of $\ell_{1}$. 
\end{proof}

It is worth pointing out that if an operator fixes a copy of $\ell_{1}$ then, in general, its adjoint does not fix a copy of $c_{0}$ \cite[Example 1.2]{Howard}.

\begin{corollary}\label{Cor:weak-compactness-p1}
	Let $g\in\mathcal{B}$ and $1\leq p<+\infty$. Then $T_g: RM(p,1)\rightarrow RM(p,1)$ is weakly compact if and only if it is compact (and if and only if $g\in \mathcal{B}_{0}$).
\end{corollary}

\end{document}